\documentclass[11pt,twoside]{article}

\usepackage[english]{babel}

\usepackage{amsmath}
\usepackage{amssymb}

\usepackage{mathrsfs}
\usepackage{amsthm}

\usepackage{indentfirst}
\usepackage{color}
\usepackage{txfonts}

\usepackage{anysize}

\usepackage[colorlinks=true,
linkcolor=blue,
citecolor=red,
urlcolor=magenta,
]{hyperref}

\allowdisplaybreaks

\newtheorem{theorem}{Theorem}[section]
\newtheorem{lemma}[theorem]{Lemma}
\newtheorem{corollary}[theorem]{Corollary}
\newtheorem{proposition}[theorem]{Proposition}

\theoremstyle{definition}
\newtheorem{remark}[theorem]{Remark}
\newtheorem{definition}[theorem]{Definition}

\numberwithin{equation}{section}

\begin{document}

\title{\bf\Large Capacitary Trapping and Brownian Escape:
A Sharp Hardy Principle for Absorbing Media
\footnotetext{\hspace{-0.35cm} 2020 {\it
Mathematics Subject Classification}. Primary 35J10, 46E35; Secondary 26D10, 31B15, 60J65.
\endgraf
{\it Key words and phrases}. absorbing diffusion, Newtonian capacity,
Brownian exit quantile, Hardy inequality, Schr\"odinger form,
Fisher information.
\endgraf
This project is partially supported by the National Natural Science
Foundation of China (Grant Nos. 12431006,
12371093, and 12501118), the Beijing Natural Science Foundation (Grant No. 1262011),
the Fundamental Research Funds for the Central Universities (Grant No. 2253200028),
and NSERC of Canada (Grant No. 202979).
}}
\author{Yiqun Chen, Jie Xiao, Dachun Yang\footnote{Corresponding author,
E-mail: \texttt{dcyang@bnu.edu.cn} /
{\color{red}\today} / Final version.}, \ Wen Yuan and
Yangyang Zhang
}
\date{}
\maketitle

\vspace{-0.8cm}

\begin{center}
\begin{minipage}{13cm}
{\small {\bf Abstract:}\quad
We study Brownian motion killed upon exiting arbitrary open sets
$\Omega\subset\mathbb R^n$, $n\in[3,\infty)\cap\mathbb N$. Let $\alpha\in(0,1]$.
The infimal radius $d_\alpha(x)$ at which the
absorber near $x$ carries an $\alpha$-fraction of the Newtonian capacity
of the ball of the same radius provides a local measure of electrostatic trapping.
We prove that this capacitary loading forces an absorption probability
of at least $c_n\alpha$ within a dimension-dependent multiple of the
diffusive time scale $d_\alpha(x)^2$. Independently, if $p\in (0,1)$ 
and $\tau_p(x)$ is the time by
which Brownian motion started at $x$ has been absorbed with probability
at least $p$, then
\begin{align*}
\int_\Omega\frac{|u\left(x\right)|^2}{\tau_p\left(x\right)}\,\mathrm d x
\le\frac{2}{p^2}\int_\Omega|\nabla u\left(x\right)|^2\,\mathrm d x
\end{align*}
for every $u\in C_{\mathrm{c}}^\infty(\Omega)$.
Together these statements give a capacitary Hardy inequality with
constant $C_n\alpha^{-2}$ on every open set, answering a problem of
Maz'ya, where $C_n$ is a positive constant and,
moreover, the exponent of $\alpha^{-2}$ is sharp.
As applications, we derive corresponding estimates for survival
probabilities, a corresponding
spectral lower bound (including a recovery of
the Maz'ya--Shubin lower bound), heat dissipation, Schr\"odinger forms,
and Fisher information.
}
\end{minipage}
\end{center}

\vspace{0.15cm}

\tableofcontents

\vspace{0.2cm}

\section{Introduction}

A diffusing particle in an absorbing medium can encounter obstacles
that are close in Euclidean distance but small from the viewpoint of
potential theory. Newtonian capacity distinguishes these situations:
in three dimensions it is proportional to the electrostatic
capacitance of a conductor. This suggests a local question that is
both geometric and dynamical. At what spatial scale does the absorber
near a particle acquire enough capacity to force a prescribed
probability of absorption, and what energy is required for a state
to concentrate where this scale is small?

We answer this question uniformly over all open sets
$\Omega\subset\mathbb R^n$, $n\in[3,\infty)\cap\mathbb N$. The closed set
$F:=\mathbb R^n\setminus\Omega$ acts as a perfect absorber. Throughout,
Brownian motion is normalized by
\begin{align}\label{eq:intro-brownian}
X_t=x+\sqrt2\,W_t
\quad\text{and}\quad
\tau_\Omega=\inf\{t\in[0,\infty):X_t\notin\Omega\},
\end{align}
so its generator is $\Delta$.
The killed heat semigroup is defined as follows.
For every bounded Borel function $f:\Omega\to\mathbb R$, extended by
zero outside $\Omega$, and any $x\in\Omega$ and $t\in[0,\infty)$, we define
\begin{align*}
P_t^\Omega f\left(x\right)
:=
\mathbb E_x\!\left[
f\left(X_t\right)\mathbf 1_{\{t<\tau_\Omega\}}
\right].
\end{align*}
It satisfies $P_t^\Omega=e^{-tA_\Omega}$, where
$A_\Omega=-\Delta_\Omega^D$ is the nonnegative Dirichlet Laplacian.
For any $u\in H_0^1(\Omega)$, the closed quadratic form of $A_\Omega$ is
\begin{align*}
\mathcal E_\Omega\left(u,u\right)=\int_\Omega|\nabla u\left(x\right)|^2\,\mathrm d x.
\end{align*}
The same form describes the kinetic energy of a quantum particle
with Dirichlet boundary conditions. Our principle links the
electrostatic size of the absorber to finite-time Brownian killing
and then to this energy, with constants independent of the geometry.

\subsection{Absorption quantiles and the capacitary trapping scale}

For $p\in(0,1)$ define the \emph{absorption-quantile time $\tau_p$} and its associated
\emph{length $\rho_p$}, respectively, by setting, for any $x\in \Omega$,
\begin{align*}
\tau_p\left(x\right):=\inf\{t\in(0,\infty):\mathbb P_x\left(\tau_\Omega\le t\right)\ge p\}
\quad\text{and}\quad \rho_p\left(x\right):=\sqrt{\tau_p\left(x\right)}.
\end{align*}
The infimum is $\infty$ when the threshold is never reached, and
$\infty^{-1}:=0$. Thus, $\tau_p(x)$ measures the time by which a
particle started at $x$ is absorbed with probability at least $p$.
It remains meaningful when the mean exit time is infinite.

Write $\operatorname{cap}$ for Newtonian capacity, with the normalization $\operatorname{cap}(B(0,r))=\kappa_n r^{n-2}$ and
\begin{align*}
\kappa_n:=n\left(n-2\right)|B\left(0,1\right)|.
\end{align*}
For $\alpha\in(0,1]$, the \emph{capacitary trapping scale} is,
for any $x\in\Omega$,
\begin{align}\label{eq:intro-trapping-scale}
d_\alpha\left(x\right):=\inf\left\{r\in(0,\infty):
\operatorname{cap}\left(\overline{F\cap B\left(x,r\right)}\right)
\ge\alpha\kappa_n r^{n-2}\right\}.
\end{align}
The capacity ratio in this definition is a dimensionless measure of
the local capacitary loading of the absorber. The scale $d_\alpha$
is the first radius at which this loading reaches $\alpha$; the
ratio itself need not be monotone in the radius. A compact set of
capacity zero is invisible to this measurement even when it is
geometrically close. In contrast, uniformly thick complements give
a scale comparable to the boundary distance, as shown in
Proposition~\ref{prop:uniform-thickness-distance-comparison}.

\subsection{The trapping principle and its sharpness}

We define
\begin{align}\label{eq:intro-constants}
a_n:=1-2^{2-n},\quad
T_n:=\frac1{4\pi}
\left(\frac{4\kappa_n}{na_n}\right)^{\frac{2}{n}},\quad
\text{and}\quad C_n:=\frac{8T_n}{a_n^2}.
\end{align}
The first theorem is a dynamical energy inequality whose constant is
independent of the dimension and the absorbing set. In what follows,
the \emph{notation $C_{\mathrm{c}}^\infty(\Omega)$} denotes
the set of all infinitely differentiable functions on $\Omega$
with compact support.

\begin{theorem}
\label{thm:exit-distance-hardy}
Let $\Omega\subset\mathbb R^n$ be open and $p\in(0,1)$. Then
\begin{align*}
\int_\Omega\frac{|u(x)|^2}{\tau_p(x)}\,\mathrm d x
=\int_\Omega\frac{|u(x)|^2}{\rho_p(x)^2}\,\mathrm d x
\le\frac{2}{p^2}\int_\Omega|\nabla u\left(x\right)|^2\,\mathrm d x
\end{align*}
for every $u\in C_{\mathrm{c}}^\infty(\Omega)$.
\end{theorem}

Observe that the conclusion of Theorem~\ref{thm:exit-distance-hardy}
is independent of the dimension $n$ of the underlying space $\mathbb R^n$.

The second theorem supplies the electrostatic input. In particular,
it converts local capacity into killing on a fixed multiple of the
natural diffusive time, rather than only into eventual hitting.

\begin{theorem}
\label{thm:compact-killing}
Let $n\in[3,\infty)\cap\mathbb N$, let
$\Omega\subset\mathbb R^n$ be open, and write
$F:=\mathbb R^n\setminus\Omega$. Let
$\kappa_n:=n(n-2)|B(\mathbf0,1)|$, and let $a_n$ and $T_n$ be the
dimensional constants in \eqref{eq:intro-constants}.
Let $x\in\Omega$, $r\in(0,\infty)$, and $\beta\in(0,1]$. Suppose that a compact
set $K\subset F\cap\overline{B(x,r)}$ satisfies
$\operatorname{cap}\left(K\right)\ge\beta\kappa_n r^{n-2}.$
Then
\begin{align}\label{eq:compact-killing}
\mathbb P_x\left(\tau_\Omega\le T_nr^2\right)\ge\frac{a_n\beta}{2},
\end{align}
or, equivalently, the survival probability satisfies
\begin{align}\label{eq:compact-survival}
P_{T_nr^2}^\Omega\mathbf1\left(x\right)\le1-\frac{a_n\beta}{2}, \quad \forall\,x\in\Omega.
\end{align}
\end{theorem}

For $K:=\overline{F\cap B(x,r)}$, this theorem yields
$\rho_{\frac{a_n\alpha}{2}}(x)\le\sqrt{T_n}\,d_\alpha(x)$ for $x\in\Omega$.
Combining this with Theorem~\ref{thm:exit-distance-hardy} gives our main
result, including the optimal dependence on the loading parameter.

\begin{theorem}\label{thm:main}
Let $n\in[3,\infty)\cap\mathbb N$ and $\alpha\in(0,1]$.
Let $C_n$ be the positive dimensional constant defined in
\eqref{eq:intro-constants}. Then the following statements hold.
\begin{itemize}
\item[\rm(i)] For every open $\Omega\subset\mathbb R^n$ and every
$u\in C_{\mathrm{c}}^\infty(\Omega)$,
\begin{align*}
\int_\Omega\frac{|u(x)|^2}{d_\alpha(x)^2}\,\mathrm d x
\le\frac{C_n}{\alpha^2}\int_\Omega|\nabla u\left(x\right)|^2\,\mathrm d x.
\end{align*}
\item[\rm(ii)] There exists $c_n\in(0,\infty)$, depending only on $n$, such that for
each $\alpha\in(0,1]$ there exist a bounded connected domain
$\Omega_\alpha$ and a nonzero $u_\alpha\in C_{\mathrm{c}}^\infty(\Omega_\alpha)$
for which
\begin{align*}
\int_{\Omega_\alpha}\frac{|u_\alpha\left(x\right)|^2}{d_{\alpha,\Omega_\alpha}\left(x\right)^2}
\,\mathrm d x
\ge\frac{c_n}{\alpha^2}
\int_{\Omega_\alpha}|\nabla u_\alpha\left(x\right)|^2\,\mathrm d x.
\end{align*}
\end{itemize}
Consequently, the best constant uniform over all open sets is
comparable to $\alpha^{-2}$. The dimensional coefficient $C_n$ is
explicit but is not asserted to be optimal.
\end{theorem}

The sharpness example for Theorem \ref{thm:main}(ii)
presented in Section \ref{sec:alpha-sharpness} has an absorber extended in $n-2$ directions
and a two-dimensional transverse annulus. Its capacity has
logarithmic scaling, and a logarithmic test profile separates the
weighted mass from the energy by a factor of order $\alpha^{-2}$.
In $n=3$ the absorber is a long filament. These examples witness
the optimality of the exponent.

\subsection{Diffusion, quantum energy, and statistical information}

The preceding theorems have several direct consequences for diffusion,
Dirichlet Schr\"odinger forms, and Fisher information, developed in
Section~\ref{sec:physical-consequences}. For diffusion coefficient
$D\in(0,\infty)$, a uniform bound $d_\alpha\le R$ gives
\begin{align*}
\mathbb P_x\left(\tau_\Omega^{\left(D\right)}>t\right)
\le\left(1-\frac{a_n\alpha}{2}\right)^{\lfloor \frac{Dt}{T_nR^2}\rfloor}.
\end{align*}
Here, and thereafter, the \emph{notation $\tau_\Omega^{(D)}$} is the exit time for
$x+\sqrt{2D}\,W_t$, and the \emph{notation $\lfloor r\rfloor$} for any $r\in\mathbb{R}$
denotes the largest integer not greater than $r$. Repeated absorption opportunities thus produce
exponential survival decay. Without a uniform bound on $d_\alpha$,
the local Hardy law still supplies a weighted space-time
heat-dissipation estimate.

Let $A_\Omega:=-\Delta_\Omega^D$ denote the nonnegative Dirichlet
Laplacian and $V_\alpha:=d_\alpha^{-2}$. The main theorem extends to
$H_0^1(\Omega)$ and, for any $u\in H_0^1(\Omega)$, gives
\begin{align*}
\int_\Omega V_\alpha\left(x\right)|u\left(x\right)|^2\,\mathrm d x
\le\frac{C_n}{\alpha^2}\mathcal E_\Omega(u,u).
\end{align*}
Consequently, for any $g\in[0,\frac{\alpha^2}{C_n}]$ and
$u\in H_0^1(\Omega)$,
\begin{align*}
\mathcal E_\Omega(u,u)
-g\int_\Omega V_\alpha\left(x\right)|u\left(x\right)|^2\,\mathrm d x
\ge0.
\end{align*}
The corresponding quadratic form is closed whenever
$g\in[0,\frac{\alpha^2}{C_n})$. The scale $\alpha^2$ of the best
geometry-independent nonnegativity threshold is optimal up to
dimensional constants.

For a probability density $f$ with $\sqrt f\in H_0^1(\Omega)$,
define its \emph{Fisher information} by
\begin{align*}
I\left(f\right):=4\int_\Omega|\nabla\sqrt f\left(x\right)|^2\,\mathrm d x.
\end{align*}
Then the energy inequalities give
\begin{align*}
\int_\Omega\frac{f\left(x\right)}{d_\alpha\left(x\right)^2}\,\mathrm d x
\le\frac{C_n}{4\alpha^2}I\left(f\right)
\quad\mathrm{and}\quad
\int_\Omega\frac{f\left(x\right)}{\tau_p\left(x\right)}\,\mathrm d x
\le\frac{1}{2p^2}I\left(f\right).
\end{align*}
A probability density concentrating in regions where the absorption
quantile is short must therefore pay a Fisher-information cost.

All constants are independent of the absorbing geometry. The
inequalities therefore apply to a random absorbing medium for each
realization, and can be averaged whenever the relevant quantities
are measurable. No stationarity or independence of the traps is
needed for this realization-wise conclusion.

\subsection{Relation to earlier work and proof structure}

Hardy inequalities originate in the work of Hardy
\cite{Hardy1920,Hardy1925} and connect Sobolev estimates
\cite{Chua2005,DominguezTikhonov2019,KalamajskaPietruskaPaluba2011},
singular-potential evolution equations
\cite{BarasGoldstein1984,VazquezZuazua2000}, and Dirichlet forms
\cite{ArendtGoldsteinGoldstein2006,Fitzsimmons2000}.
Further developments include
\cite{AghajaniKinnunenRadulescu2025,GunawanHakimNakaiSawano2018,
KinnunenMartio1997}. For a suitable domain $\Omega$ and any
$u\in C_{\mathrm{c}}^\infty(\Omega)$, the classical distance
$p$-Hardy inequality reads
\begin{align}\label{eq:intro-classical-p-hardy}
\int_\Omega\frac{|u(x)|^p}{\operatorname{dist}(x,\partial\Omega)^p}
\,\mathrm d x
\lesssim\int_\Omega|\nabla u\left(x\right)|^p\,\mathrm d x.
\end{align}
Ne\v{c}as \cite{Necas1962} established such domain inequalities in
the Lipschitz setting. Ancona \cite{AnconaHardy}, Lewis
\cite{LewisUniformlyFat}, and Wannebo \cite{Wannebo1990} showed that
capacity thickness can replace boundary regularity.
Sharp and fractional variants of Hardy inequalities were developed,
among others, by Frank and Seiringer \cite{fs08}, while
Dyda and Frank \cite{df12} established a fractional
Hardy--Sobolev--Maz'ya inequality for arbitrary domains.

For $p=2$, Maz'ya's capacitary characterization states that
\begin{align*}
\int_\Omega\frac{|u\left(x\right)|^2}{\operatorname{dist}(x,\partial\Omega)^2}
\,\mathrm d x\lesssim\int_\Omega|\nabla u\left(x\right)|^2\,\mathrm d x
\end{align*}
is equivalent to an estimate of the weighted measure of each
compact $K\Subset\Omega$ by its relative capacity
$\operatorname{cap}_\Omega(K)$; see \cite[pp.~733--735]{MazyaSobolev} and
\cite[Theorem~2.1]{KinnunenKorteHardy}. Maz'ya
\cite[Problem~8, Section~3.8]{MazyaProblems} asked whether replacing
boundary distance by the capacitary distance makes the inequality
valid on every domain. Theorem~\ref{thm:main} answers this question
and gives the optimal power of $\alpha^{-2}$. The compact-closure
convention in \eqref{eq:intro-trapping-scale} is explained in
Remark~\ref{rem:literal-mazya-capacity}; the interpretation using
outer capacity also follows from our result.

Capacity has a longstanding role in spectral estimates. Maz'ya
and Shubin \cite{MazyaShubinDrum} relate the bottom of the
Dirichlet spectrum to a single interior capacitary radius.
Their lower bound strengthens an earlier estimate of Lieb
\cite{Lieb1983}.
In Lieb's result, the relevant geometric condition is formulated in
terms of the Lebesgue measure of the complement of the domain
inside a ball rather than capacity.
Frank and Larson \cite{FrankLarson2021} rederive Lieb's 
estimate from Davies' Hardy
inequality and ask in \cite[Section 2 Remark (3)]{FrankLarson2021} 
whether a strengthened Hardy inequality could
similarly imply the capacitary result of Maz'ya and Shubin. Our
capacitary Hardy inequality gives such an implication directly, and
Corollary~\ref{cor:mazya-shubin-lower} below recovers the corresponding
spectral lower bound with linear dependence on the capacitary loading
parameter.
Lu and Steinerberger \cite{LuSteinerbergerDV} use exit quantiles in
principal-eigenvalue estimates. Our absorption-time inequality
controls the full spatially varying quantile as a multiplication
weight. The distinction is consequential: the sharp local
$\alpha^{-2}$ law coexists with global survival and spectral-gap
bounds of order $\alpha$ under uniform capacitary loading.
The Brownian and Schr\"odinger connection is treated systematically
in \cite{ChungZhao}.

The proof of Theorem \ref{thm:main}(i) can be separated into
two main steps. A variable-time semigroup estimate
controls the time-normalized change of a state by its energy. Applied
to capacitary test functions for the killed Brownian semigroup, it
gives Theorem~\ref{thm:exit-distance-hardy} through Maz'ya's
isocapacitary criterion. Green equilibrium measures then yield a
capacity lower bound for eventual hitting inside a larger ball.
A heat-kernel estimate bounds late hitting as a fraction of this
same eventual probability, giving Theorem~\ref{thm:compact-killing}
with a time constant independent of $\alpha$.

Section~\ref{sec:notation} records capacity conventions, measurability
and thickness comparisons. Section~\ref{sec:euclidean-variable}
proves the semigroup estimate, and Section~\ref{sec:brownian-exit}
establishes the two absorption theorems and the capacitary Hardy
bound. Section~\ref{sec:alpha-sharpness} constructs the codimension-two
sharpness examples. Section~\ref{sec:physical-consequences} develops
the diffusion, quantum and information consequences.

\subsection{Notation} 

We use the following conventions.
We always denote by $C$ a positive constant which is
independent of the main parameters involved, but it may vary from line to
line. We use $C_{\alpha,\beta,\ldots}$ to denote a positive constant
depending on the indicated parameters $\alpha,\beta,\ldots$.
The symbol $C_n$ is reserved for the explicit coefficient in
\eqref{eq:intro-constants}; unspecified dimensional constants in
auxiliary estimates are denoted by $M_n$ and may vary from line to line.
The notation $f\lesssim g$ means that $f\le Cg$. If $f\lesssim g$ and $g\lesssim f$, we
write $f\sim g$.
For $x,y\in\mathbb R^n$, we write
$x\cdot y$ and $|x|$ for the Euclidean inner product and norm,
respectively. For a differentiable function $u$, $\nabla u$ denotes
its Euclidean gradient, and for a twice differentiable function $u$,
$\Delta u:=\sum_{i=1}^n\partial_{x_i}^2u.$
For $x\in\mathbb R^n$ and $r\in(0,\infty)$, let
\begin{align*}
B\left(x,r\right):=\{y\in\mathbb R^n:|y-x|<r\}
\end{align*} and $\overline{B(x,r)}$
be its closure.
For a Lebesgue measurable set $A\subset\mathbb R^n$,
$|A|$ denotes its Lebesgue measure.
For $A,B\subset\mathbb R^n$, define
\begin{align*}
\operatorname{dist}\left(A,B\right):=\inf\{|a-b|:a\in A,\ b\in B\}
\quad\mathrm{and}\quad
\operatorname{dist}\left(x,B\right):=\operatorname{dist}\left(\{x\},B\right),
\end{align*}
with the convention that the distance is $\infty$ if one of the
sets is empty. If $U\subset\mathbb R^n$ is open, we write
the \emph{notation $A\Subset U$} when $\overline A$ is compact and
$\overline A\subset U$.
Throughout the article, we write $F:=\Omega^\complement$ whenever
an open set $\Omega\subset\mathbb R^n$ is under consideration.
For any set $A$, $\mathbf 1_A$ denotes its characteristic function, while
$\mathbf 1$ denotes the constant function equal to one on the space
under consideration. For $s\in\mathbb R$ and
$a,b\in[-\infty,\infty]$, let $s_+:=\max\{s,0\}$ and $a\wedge b:=\min\{a,b\}$.
For any $a\in\mathbb R^n$, $s\in(0,\infty)$, and $A\subset\mathbb R^n$, we write
$a+sA:=\{a+sy:y\in A\}$. For $a\in[0,\infty)$, the \emph{notation} $s\to a^+$ 
(resp. $s\to a^-$) means that
$s\in(a,\infty)$ [resp. $s\in(-\infty,a)$] and $s\to a$. 
Finally, in all subsequent proofs we
retain the notation introduced in the relevant statement.

\section{Capacities and the capacitary distance}
\label{sec:notation}

In this section, we introduce the capacities and the capacitary distance
and establish their basic properties, respectively, in Subsections \ref{s2.1}
and \ref{s2.2}.

\subsection{Capacities and basic notation\label{s2.1}}

Throughout this article, we always let $n\in[3,\infty)\cap\mathbb N$.
For any open set $U\subset\mathbb R^n$, the \emph{notation 
$C_{\rm{c}}^\infty(U)$} denotes the space of
all real-valued smooth functions compactly supported in $U$, and
the \emph{notation $C_{\rm{c}}^{0,1}(\mathbb R^n)$} denotes the space
of all compactly supported globally
Lipschitz real-valued functions on $\mathbb R^n$.  For
$v\in C_{\rm{c}}^{0,1}(U)$,
$\nabla v$ denotes its almost-everywhere
classical gradient, which agrees with its weak gradient.
The \emph{Sobolev space} $H^1(U)$ consists of all functions
$u\in L^2(U)$ whose distributional gradient
$\nabla u\in L^2(U;\mathbb R^n)$, and let
\begin{align*}
\|u\|_{H^1\left(U\right)}
:=
\left[
\|u\|_{L^2\left(U\right)}^2+
\|\nabla u\|_{L^2\left(U;\mathbb R^n\right)}^2
\right]^{\frac12}.
\end{align*}
The space $H_0^1(U)$ is the closure of $C_{\rm{c}}^\infty(U)$ in $H^1(U)$.
Unless a complex space is explicitly specified, these spaces are real.
Their complex versions are obtained by allowing real and imaginary
parts in the corresponding real spaces.
For $v\in C_{\rm{c}}^{0,1}(\mathbb R^n)$, Rademacher's theorem gives
$v\in H^1(\mathbb R^n)$.
A nonnegative Borel measure on an open set $U$ is said to be \emph{locally
finite} if it is finite on every compact subset of $U$, and \emph{Radon} if,
in addition, it is inner regular.

We first record the two mollification facts used below.
\begin{lemma}
\label{lem:mollifier-cutoff-package}
Let $\rho\in C_{\rm{c}}^\infty(B(\mathbf0,1))$ with
$\rho\ge0$ and
$\int_{\mathbb R^n}\rho(x)\,\mathrm d x=1,$
and, for $\varepsilon\in(0,\infty)$, let
$\rho_\varepsilon
:=
\varepsilon^{-n}\rho(\frac{\cdot}{\varepsilon}).$
Then the following assertions hold.
\begin{itemize}
\item[{\rm (i)}] If $v\in H^1(\mathbb R^n)$, then
$\rho_\varepsilon*v\longrightarrow v$
in $H^1(\mathbb R^n)$
as $\varepsilon\to0^+$.
\item[{\rm (ii)}] Let $U\subset\mathbb R^n$ be open and $V\subset\mathbb R^n$ satisfy
$\overline V\subset U$.  If $v$ is locally integrable and equal to a
constant $c$ almost everywhere on $U$, then
$\rho_\varepsilon*v=c$ on $V$
whenever
$\varepsilon\in
(0,\operatorname{dist}(\overline V,\mathbb R^n\setminus U)).$
\end{itemize}
\end{lemma}

\begin{proof}
Item~\textup{(i)} is the standard $H^1$-mollification theorem; see,
for instance, \cite[Chapter~2, Section~2.1]{ZiemerSobolev}.

For \textup{(ii)}, if $x\in V$ and
$z\in\operatorname{supp}\rho_\varepsilon$, then $|z|<\varepsilon$, and hence
$x-z\in U$.  Therefore,
\begin{align*}
\left(\rho_\varepsilon*v\right)\left(x\right)
=
c\int_{\mathbb R^n}\rho_\varepsilon\left(z\right)\,\mathrm d z
=
c.
\end{align*}
This completes the proof of Lemma \ref{lem:mollifier-cutoff-package}.
\end{proof}

For any compact set $E\subset\mathbb R^n$, we define its
\emph{Newtonian capacity} by
\begin{align*}
\operatorname{cap}\left(E\right):=
\inf\left\{
\int_{\mathbb R^n}|\nabla\varphi\left(x\right)|^2\,\mathrm d x:
\varphi\in C_{\rm{c}}^\infty\left(\mathbb R^n\right),\
\varphi\ge1\text{ on }E
\right\}.
\end{align*}
For any open set $O\subset\mathbb R^n$, we define
\begin{align*}
\operatorname{cap}\left(O\right):=\sup\{\operatorname{cap}\left(K\right):K\subset O,\ K\text{ is compact}\}.
\end{align*}
For any  set $A\subset\mathbb R^n$, we define its \emph{outer
Newtonian capacity} by
\begin{align*}
\operatorname{cap}\left(A\right):=\inf\{\operatorname{cap}\left(O\right):A\subset O,\ O\subset\mathbb R^n
\text{ open}\}.
\end{align*}
Here, and thereafter, we define $\sup\emptyset:=0$ and $\inf\emptyset:=\infty$.
For any open set $\Omega\subset\mathbb R^n$ and a compact set $K\Subset \Omega$, we
define the \emph{capacity of $K$ relative to $\Omega$} by
\begin{align*}
\operatorname{cap}_\Omega\left(K\right):=
\inf\left\{
\int_\Omega|\nabla\varphi\left(x\right)|^2\,\mathrm d x:
\varphi\in C_{\rm{c}}^\infty\left(\Omega\right),\
\varphi\ge1\text{ on }K
\right\}.
\end{align*}

The following lemma are some basic properties of capacity; see, for instance, \cite[Theorem~4.15]{EvansGariepy}.
\begin{lemma}
\label{lem:capacity-package}
The following statements hold.
\begin{itemize}
\item[{\rm (i)}] If
$A\subset B\subset\mathbb R^n$, then $\operatorname{cap}(A)\le\operatorname{cap}(B)$.
\item[{\rm (ii)}] If
$K_1\supset K_2\supset\cdots$ are compact, then
\begin{align*}
\operatorname{cap}\!\left(\bigcap_{j=1}^\infty K_j\right)
=\lim_{j\to\infty}\operatorname{cap}\left(K_j\right).
\end{align*}
If $A_1\subset A_2\subset\cdots$ are Borel, then
\begin{align*}
\operatorname{cap}\!\left(\bigcup_{j=1}^\infty A_j\right)
=\lim_{j\to\infty}\operatorname{cap}\left(A_j\right).
\end{align*}
\item[{\rm (iii)}] For any $a\in\mathbb R^n$,
$s\in(0,\infty)$, and any Borel set $A\subset\mathbb R^n$,
$\operatorname{cap}(a+sA)=s^{n-2}\operatorname{cap}(A).$
\item[{\rm (iv)}] If
$D\subset\mathbb R^n$ is open and $E\Subset D$ is compact, then
$\operatorname{cap}_D(E)\ge\operatorname{cap}(E).$
\end{itemize}
\end{lemma}

The following lemma follows from \cite[Section~2.2.4, p.~148]{MazyaSobolev}.

\begin{lemma}\label{1709}
For any $x\in\mathbb R^n$ and $r\in(0,\infty)$, it holds
\begin{align*}
\operatorname{cap}\left(B\left(x,r\right)\right)
=\operatorname{cap}\left(\overline{B\left(x,r\right)}\right)
=\kappa_n r^{n-2}
\quad and\quad
\kappa_n:=n\left(n-2\right)|B\left(\mathbf0,1\right)|.
\end{align*}
\end{lemma}

\subsection{The capacitary distance\label{s2.2}}

For any open set $\Omega\subset\mathbb R^n$, $F:=\Omega^\complement$,
$\alpha\in(0,\infty)$, and $x\in\Omega$, we define the \emph{capacitary distance}
\begin{align}\label{eq:dalpha-def}
d_{\alpha,\Omega}\left(x\right)
:=
\inf\left\{r\in\left(0,\infty\right):
\operatorname{cap}\left(\overline{F\cap B\left(x,r\right)}\right)
\ge\alpha\operatorname{cap}\left(B\left(\mathbf0,r\right)\right)\right\}.
\end{align}
When there is no ambiguity, we simply write $d_\alpha$ instead of
$d_{\alpha,\Omega}$.

\begin{remark}\label{rem:literal-mazya-capacity}
In Problem~8 of \cite{MazyaProblems}, Maz'ya writes the capacitary
distance in the form
\begin{align}\label{912}
d_\alpha\left(x\right)
:=
\inf\left\{
r\in(0,\infty):
\operatorname{cap}\left(F\cap B\left(x,r\right)\right)
\ge
\alpha\operatorname{cap}\left(B\left(\mathbf0,r\right)\right)
\right\},
\end{align}
without taking the closure of $F\cap B(x,r)$.
On the other hand, in \cite{MazyaProblems}
Maz'ya defines the Wiener capacity for any compact set.
Since $F\cap B(x,r)$ need not be closed, the capacity appearing in
\cite[Problem~8]{MazyaProblems} therefore requires an interpretation.  There are two natural
ways to interpret it.  First, one may interpret the symbol $\operatorname{cap}$
in \eqref{912} as the outer Newtonian capacity.  Alternatively,
one may use the Wiener capacity and read it
literally for an arbitrary bounded set $A\subset\mathbb R^n$, which gives
\begin{align*}
\operatorname{cap}_{\mathrm{var}}\left(A\right):=
\inf\left\{
\int_{\mathbb R^n}|\nabla v\left(x\right)|^2\,\mathrm d x:
v\in C_{\rm{c}}^{0,1}\left(\mathbb R^n\right),\
v\ge1\text{ on }A
\right\}.
\end{align*}
Now, we claim that
\begin{align}\label{eq:literal-cap-closure}
\operatorname{cap}_{\mathrm{var}}\left(A\right)=\operatorname{cap}\left(\overline A\right).
\end{align}
Indeed, the continuity gives $\operatorname{cap}_{\mathrm{var}}(A)=\operatorname{cap}_{\mathrm{var}}(\overline A)$, and the inclusion
$C_{\rm c}^\infty(\mathbb R^n)\subset C_{\rm c}^{0,1}(\mathbb R^n)$ yields
$\operatorname{cap}_{\mathrm{var}}(\overline A)\le\operatorname{cap}(\overline A)$.  Conversely, let
$v\in C_{\rm c}^{0,1}(\mathbb R^n)$ satisfy $v\ge1$ on $\overline A$ and
fix $\delta\in(0,1)$.  Let
\begin{align*}
\Phi_\delta\left(t\right):=
\min\left\{1,\frac{\max\{t,0\}}{1-\delta}\right\},
\end{align*}
and $w_\delta:=\Phi_\delta\circ v$.
The function
$\Phi_\delta$ is Lipschitz, satisfies $\Phi_\delta(0)=0$, and
$\operatorname{Lip}(\Phi_\delta)=(1-\delta)^{-1}.$
Hence, by the Sobolev chain rule (see, for instance,
\cite[Chapter~2, Sections~2.1--2.2]{ZiemerSobolev}),
$w_\delta\in H^1(\mathbb R^n)$ and
$|\nabla w_\delta|
\le (1-\delta)^{-1}|\nabla v|$
a.e. on $\mathbb R^n$.
Moreover, since $\Phi_\delta(0)=0$, it follows that
$\operatorname{supp} w_\delta\subset\operatorname{supp} v$ and hence $\operatorname{supp} w_\delta$ is compact,
which further implies that
\begin{align}\label{857}
\int_{\mathbb R^n}|\nabla w_\delta\left(x\right)|^2\,\mathrm d x
\le\left(1-\delta\right)^{-2}
\int_{\mathbb R^n}|\nabla v\left(x\right)|^2\,\mathrm d x.
\end{align}
Moreover, $w_\delta=1$ on the open set
$U_\delta:=\{x:v(x)>1-\delta\}$, which contains $\overline A$.
Choose an open set
$V_\delta$ with $\overline A\subset V_\delta\Subset U_\delta$.
Then $\operatorname{dist}(\overline{V_\delta},U_\delta^\complement)>0$.  For every sufficiently small
$\varepsilon\in(0,\infty)$,
Lemma \ref{lem:mollifier-cutoff-package}(ii) shows that
$\rho_\varepsilon*w_\delta=1$ on $V_\delta$, while
Lemma \ref{lem:mollifier-cutoff-package}\textup{(i)}
shows that $\rho_\varepsilon*w_\delta$ converges to $w_\delta$ as
$\varepsilon\to0^+$ in $H^1(\mathbb R^n)$.
Letting $\varepsilon\to0^+$ and using the $H^1$ convergence and \eqref{857},
we obtain
\begin{align*}
\operatorname{cap}\left(\overline A\right)
&\le\lim_{\varepsilon\to0}
\int_{\mathbb R^n}|\nabla \left(\rho_\varepsilon*w_\delta
\right)\left(x\right)|^2\,\mathrm d x\\
&=\int_{\mathbb R^n}|\nabla w_\delta\left(x\right)|^2\,\mathrm d x
\le\left(1-\delta\right)^{-2}\int_{\mathbb R^n}|\nabla v\left(x\right)|^2\,\mathrm d x.
\end{align*}
Letting $\delta\to0^+$ and taking the infimum over $v$, we find that
the above claim \eqref{eq:literal-cap-closure} holds.

In this article, we use the closure in the definition of the capacitary distance for two
reasons. First, by \eqref{eq:literal-cap-closure}, the literal
interpretation of Maz'ya's Wiener capacity gives
$\operatorname{cap}_{\mathrm{var}}(F\cap B(x,r))
=\operatorname{cap}(\overline{F\cap B(x,r)})$.
Thus, our definition agrees exactly with this literal interpretation. The lower-bound construction in
Section~\ref{sec:alpha-sharpness} places compact absorbers inside
$F\cap B(x,r)$, so the sharpness statement also applies to the
outer-capacity interpretation.
Second, if one instead interprets the capacity in \eqref{912} as the
outer Newtonian capacity, then, by the monotonicity of the capacity,
we obtain $\operatorname{cap}(F\cap B(x,r))
\le\operatorname{cap}(\overline{F\cap B(x,r)})$.
It follows that the corresponding capacitary distance is no smaller than the one
used in this article, and hence the associated Hardy inequality is
weaker and follows from the one proved in this article. For these reasons,
we use $\operatorname{cap}(\overline{F\cap B(x,r)})$ in the definition of
the capacitary distance.
\end{remark}

\begin{lemma}
\label{lem:dalpha-elementary-properties}
Let $\Omega$, $\Omega_1$, and $\Omega_2$ be open subsets of
$\mathbb R^n$.
Let $\alpha,\alpha_1,\alpha_2\in(0,\infty)$.  The following assertions hold.
\begin{itemize}
\item[\rm (i)] If $\Omega_1\subset\Omega_2$, then, for every
$x\in\Omega_1$,
$d_{\alpha,\Omega_1}(x)\le d_{\alpha,\Omega_2}(x).$
\item[\rm (ii)] If $\alpha_1\in(0,\alpha_2]$, then, for every $x\in\Omega$,
$d_{\alpha_1}(x)\le d_{\alpha_2}(x).$
\item[\rm (iii)] For every $x\in\Omega$,
$\operatorname{dist}(x,\Omega^\complement)\le d_\alpha(x).$
\item[\rm (iv)] If $\alpha\in(1,\infty)$, then $d_\alpha\equiv\infty$.
\end{itemize}
\end{lemma}

\begin{proof}
Items~\textup{(i)} and~\textup{(ii)} follow immediately from the definition
and the monotonicity of the Newtonian capacity.
To prove \textup{(iii)}, fix
$x\in\Omega$.  If $r\in(0,\operatorname{dist}(x,F))$, then $F\cap B(x,r)=\emptyset$ and hence
$\operatorname{cap}(\overline{F\cap B(x,r)})=0.$
Thus, no such $r$ is admissible in \eqref{eq:dalpha-def}, which gives
$\operatorname{dist}(x,F)\le d_\alpha(x)$.

Finally, if $\alpha\in(1,\infty)$, then, by monotonicity and translation invariance of the capacity,
no radius is admissible in \eqref{eq:dalpha-def}. In this case,
$d_\alpha\equiv\infty$.
This completes the proof of Lemma \ref{lem:dalpha-elementary-properties}.
\end{proof}

Now, we consider the measurability of the capacitary distance.
The following lemma is the Euclidean case of the measurable projection
theorem; see, for instance,
Cohn \cite[Proposition~8.4.4, p.~264]{CohnMeasureTheory}.

\begin{lemma}
\label{lem:borel-projection-measurable}
Let $A\subset\mathbb R^n\times(0,\infty)$ be a Borel set.
Then its projection onto $\mathbb R^n$ is Lebesgue measurable.
\end{lemma}

\begin{lemma}
\label{lem:borel-projection-measurable2}
Let $\Omega\subset\mathbb R^n$ be an open set.  For every
$\alpha\in(0,\infty)$, $d_\alpha$ is Lebesgue measurable on $\Omega$.
\end{lemma}

\begin{proof}
For $(x,r)\in\mathbb R^n\times(0,\infty)$, let
$c(x,r):=\operatorname{cap}(\overline{F\cap B(x,r)}).$
We first show that $c$ is Borel measurable.
To this end, equip $C_{\rm{c}}^\infty(\mathbb R^n)$ with the metric
\begin{align*}
d\left(\phi,\psi\right):=
\|\phi-\psi\|_{L^\infty}
+\|\nabla\phi-\nabla\psi\|_{L^2}.
\end{align*}
Endow $C_{\rm{c}}(\mathbb R^n)\times L^2(\mathbb R^n;\mathbb R^n)$
with the norm
$\|(g,G)\|:=\|g\|_{L^\infty}+\|G\|_{L^2}$.
The map
\begin{align*}
T:C_{\rm{c}}^\infty\left(\mathbb R^n\right)\longrightarrow
C_{\rm{c}}\left(\mathbb R^n\right)\times L^2\left(\mathbb R^n;\mathbb R^n\right),
\end{align*}
defined by setting $T\phi:=(\phi,\nabla\phi)$
for $\phi\in C_{\rm{c}}$, is an isometric embedding.
Since both $C_{\rm{c}}(\mathbb R^n)$ and
$L^2(\mathbb R^n;\mathbb R^n)$ are separable,
it follows that the product space is separable,
and hence $C_{\rm{c}}^\infty(\mathbb R^n)$ is
separable for the metric $d$.
Choose a countable dense sequence $\{\phi_j\}_{j\in\mathbb N}$, and, for any $j\in\mathbb N$, let
$e_j:=\int_{\mathbb R^n}|\nabla\phi_j\left(x\right)|^2\,\mathrm d x$.

We claim that, for every $(x,r)\in\mathbb R^n\times(0,\infty)$,
\begin{align}\label{eq:countable-closure-cap}
c\left(x,r\right)
=
\inf\{e_j:\,j\in\mathbb N,\ \phi_j>1\text{ on }F\cap B\left(x,r\right)\}.
\end{align}
Indeed, for any $j\in\mathbb N$ such that
$\phi_j>1$ on $F\cap B(x,r)$, the continuity of $\phi_j$ implies that
$\phi_j\ge1$ on $\overline{F\cap B(x,r)}$.
Thus, $\phi_j$ satisfies the condition in the definition of
$c(x,r)$, and hence
\begin{align*}
c\left(x,r\right)
\le
\inf\{e_j:\,j\in\mathbb N,\ \phi_j>1\text{ on }F\cap B\left(x,r\right)\}.
\end{align*}
Conversely, let
$\phi\in C_{\rm{c}}^\infty(\mathbb R^n)$ satisfy
$\phi\ge1$ on $\overline{F\cap B(x,r)}$,
let $\delta,\eta\in(0,\infty)$, and let
$f:=(1+\delta)\phi$.
By the density in $C_{\rm c}^\infty(\mathbb R^n)$ of
$\{\phi_j\}_{j\in\mathbb N}$ according to the metric $d$, we conclude that
there exists
$j\in\mathbb N$ such that
$\|\phi_j-f\|_{L^\infty}<\frac{\delta}{2}$ and
$\|\nabla\phi_j-\nabla f\|_{L^2}<\eta,$
which further implies that
\begin{align*}
\|\nabla\phi_j\|_{L^2}
\le \left(1+\delta\right)\|\nabla\phi\|_{L^2}+\eta
\end{align*}
and, for every $y\in F\cap B(x,r)$,
$\phi_j(y)
\ge(1+\delta)\phi(y)-\frac{\delta}{2}>1.$
It follows that
\begin{align*}
\inf\{e_j:\,j\in\mathbb N,\ \phi_j>1\text{ on }F\cap B\left(x,r\right)\}
\le
\left[\left(1+\delta\right)\|\nabla\phi\|_{L^2}+\eta\right]^2.
\end{align*}
Letting $\eta\to0^+$ and $\delta\to0^+$ and then taking the
infimum over all admissible $\phi$, we conclude that
\begin{align*}
c\left(x,r\right)
\ge
\inf\{e_j:\, j\in\mathbb N,\ \phi_j>1\text{ on }F\cap B\left(x,r\right)\}.
\end{align*}
Combining the two inequalities proves \eqref{eq:countable-closure-cap}.

For any $j\in\mathbb N$, let
$C_j:=F\cap\{y\in\mathbb R^n:\phi_j(y)\le1\}.$
Then $C_j$ is closed, and $\phi_j>1$ on $F\cap B(x,r)$ if and only if
$\operatorname{dist}(x,C_j)\ge r$, where
$\operatorname{dist}(x,\emptyset):=\infty$.
From this and \eqref{eq:countable-closure-cap}, we infer that,
for any $a\in[0,\infty)$,
\begin{align*}
\{\left(x,r\right):c\left(x,r\right)<a\}
&=
\bigcup_{\{j\in\mathbb N:e_j<a\}}
\{\left(x,r\right):\phi_j>1\text{ on }F\cap B\left(x,r\right)\}\\
&=
\bigcup_{\{j\in\mathbb N:e_j<a\}}
\{\left(x,r\right):\operatorname{dist}\left(x,C_j\right)\ge r\}.
\end{align*}
The right-hand side is a countable union of closed sets in
$\mathbb R^n\times(0,\infty)$, and hence $c$ is Borel measurable.

For any $t\in(0,\infty)$,
$\{x\in\Omega:d_\alpha(x)<t\}$ is the projection onto
$\mathbb R^n$ of the Borel set
\begin{align*}
\left\{
\left(x,r\right)\in\Omega\times\left(0,t\right):
c\left(x,r\right)\ge\alpha\kappa_n r^{n-2}
\right\}.
\end{align*}
By Lemma~\ref{lem:borel-projection-measurable}, we find that
this projection is Lebesgue measurable
and hence $d_\alpha$ is Lebesgue measurable on $\Omega$,
which completes the proof of Lemma~\ref{lem:borel-projection-measurable2}.
\end{proof}

Next, we compare the capacitary distance with the classical boundary
distance under the standard uniform thickness assumption. We say that
$F=\Omega^\complement$ is \emph{uniformly $2$-thick} if there exists
$c_{\rm th}\in(0,1]$ such that, for every $y\in F$ and $r\in(0,\infty)$,
\begin{align}\label{eq:uniform-2-thickness}
\operatorname{cap}_{B\left(y,2r\right)}\left(F\cap\overline{B\left(y,r\right)}\right)
\ge c_{\rm th}
\operatorname{cap}_{B\left(y,2r\right)}\left(\overline{B\left(y,r\right)}\right);
\end{align}
see, for example, \cite{AnconaHardy,LewisUniformlyFat,Wannebo1990}.

\begin{proposition}\label{prop:uniform-thickness-distance-comparison}
Let $\Omega\subset\mathbb R^n$ be open and write
$F:=\mathbb R^n\setminus\Omega$.
Assume that $F$ is uniformly $2$-thick.
Then there exists $\alpha_0\in(0,1]$ such that, for every
$\alpha\in(0,\alpha_0]$ and $x\in\Omega$,
\begin{align}\label{eq:uniform-thickness-dalpha-comparison}
\operatorname{dist}\left(x,F\right)\le d_\alpha\left(x\right)\le3\operatorname{dist}\left(x,F\right).
\end{align}
\end{proposition}

\begin{proof}
If $F=\emptyset$, both distances in
\eqref{eq:uniform-thickness-dalpha-comparison} are infinite, so the
assertion is immediate. We therefore assume that $F\ne\emptyset$.
We first claim that there exists a constant $M_n\in(0,\infty)$ such that,
for any $y\in\mathbb R^n$, $r\in(0,\infty)$, and any compact
$E\subset\overline{B(y,r)}$,
\begin{align}\label{eq:relative-absolute-capacity-comparison}
\operatorname{cap}_{B\left(y,2r\right)}\left(E\right)\le M_n\operatorname{cap}\left(E\right).
\end{align}
Indeed, let $\eta\in C_{\rm{c}}^\infty(B(y,2r))$ satisfy
$\eta=1$ on $\overline{B(y,r)}$ and
$|\nabla\eta|\lesssim r^{-1}$.  If
$\varphi\in C_{\rm{c}}^\infty(\mathbb R^n)$ satisfies
$\varphi\ge1$ on $E$, then
$\eta\varphi\in C_{\rm{c}}^\infty(B(y,2r))$ and
$\eta\varphi=\varphi\ge1$ on $E$. Hence, $\eta\varphi$ satisfies the condition in the definition
of $\operatorname{cap}_{B(y,2r)}(E)$.  By H\"older's
inequality and the Sobolev inequality, we obtain
\begin{align*}
r^{-2}\int_{B\left(y,2r\right)}|\varphi\left(x\right)|^2\,\mathrm d x
\lesssim
\left(\int_{\mathbb R^n}|\varphi\left(x\right)|^{\frac{2n}{n-2}}\,\mathrm d x\right)^{\frac{n-2}{n}}
\lesssim
\int_{\mathbb R^n}|\nabla\varphi\left(x\right)|^2\,\mathrm d x,
\end{align*}
and hence
\begin{align*}
\int_{B\left(y,2r\right)}|\nabla\left(\eta\varphi\right)\left(x\right)|^2\,\mathrm d x
&\lesssim
\int_{B\left(y,2r\right)}\eta\left(x\right)^2|\nabla\varphi\left(x\right)|^2\,\mathrm d x
+r^{-2}\int_{B\left(y,2r\right)}|\varphi\left(x\right)|^2\,\mathrm d x\\
&\lesssim
\int_{\mathbb R^n}|\nabla\varphi\left(x\right)|^2\,\mathrm d x.
\end{align*}
Taking the infimum, we conclude that the above claim
\eqref{eq:relative-absolute-capacity-comparison} holds.

Fix $x\in\Omega$ and $y\in F$ such
that $|x-y|=\operatorname{dist}(x,F)$. Let
$\delta:=\operatorname{dist}(x,F)$.
From
\eqref{eq:relative-absolute-capacity-comparison},
\eqref{eq:uniform-2-thickness}, and
Lemmas~\ref{lem:capacity-package}(iv) and \ref{1709}, we infer that
\begin{align*}
M_n\operatorname{cap}\left(F\cap\overline{B\left(y,\delta\right)}\right)
&\ge \operatorname{cap}_{B\left(y,2\delta\right)}\left(F\cap\overline{B\left(y,\delta\right)}\right)
\ge c_{\rm th}\operatorname{cap}_{B\left(y,2\delta\right)}\left(\overline{B\left(y,\delta\right)}\right)\\
&\ge c_{\rm th}\operatorname{cap}\left(\overline{B\left(y,\delta\right)}\right)
\ge c_{\rm th}\kappa_n\delta^{n-2}.
\end{align*}
Since
$F\cap\overline{B(y,\delta)}\subset \overline{F\cap B(x,3\delta)}$,
it follows that
\begin{align*}
\operatorname{cap}\left(\overline{F\cap B\left(x,3\delta\right)}\right)
\ge \frac{c_{\rm th}}{M_n}\kappa_n\delta^{n-2}
=\frac{c_{\rm th}}{M_n}3^{2-n}\operatorname{cap}\left(B\left(\mathbf0,3\delta\right)\right).
\end{align*}
With $\alpha_0:=\min\{1,\frac{c_{\rm th}}{M_n}3^{2-n}\}$, this gives
$d_\alpha(x)\le3\delta=3\operatorname{dist}(x,F)$ for every
$\alpha\in(0,\alpha_0]$.  The lower bound in
\eqref{eq:uniform-thickness-dalpha-comparison} follows from
Lemma~\ref{lem:dalpha-elementary-properties}\textup{(iii)}.
This completes the proof of Proposition
\ref{prop:uniform-thickness-distance-comparison}.
\end{proof}

\begin{remark}
Let $F=\Omega^\complement$ be uniformly $2$-thick. Then, by
Proposition~\ref{prop:uniform-thickness-distance-comparison},
we find that, for every
sufficiently small $\alpha\in(0,\infty)$, Theorem~\ref{thm:main}(i)
recovers the classical Hardy inequality
\eqref{eq:intro-classical-p-hardy} in the case $p=2$.
\end{remark}

The following is the $p=q=2$ case of Maz'ya's classical
isocapacitary characterization of Hardy inequalities in
\cite[Theorem~8.5]{MazyaIsocapacitary}.

\begin{lemma}
\label{thm:isocap}
Let $\Omega\subset\mathbb R^n$ be open.
Let $\mu$ be a nonnegative locally finite Borel measure on $\Omega$.
\begin{itemize}
\item[{\rm (i)}] If there exists a positive constant $C_1$ such that,
for all $u\in C_{\rm{c}}^\infty(\Omega)$,
\begin{align*}
\int_\Omega|u\left(x\right)|^2\,\mathrm d\mu\left(x\right)
\le C_1\int_\Omega|\nabla u\left(x\right)|^2\,\mathrm d x,
\end{align*}
then, for any compact $K\Subset\Omega$,
$\mu(K)\le C_1\operatorname{cap}_\Omega(K).$
\item[{\rm (ii)}] Conversely, if there exists a constant $C_2\in[0,\infty)$ such that, for any compact $K\Subset\Omega$,
\begin{align*}
\mu\left(K\right)\le C_2\operatorname{cap}_\Omega\left(K\right),
\end{align*}
then, for all $u\in C_{\rm{c}}^\infty(\Omega)$,
\begin{align*}
\int_\Omega|u\left(x\right)|^2\,\mathrm d\mu\left(x\right)
\le4C_2\int_\Omega|\nabla u\left(x\right)|^2\,\mathrm d x
.
\end{align*}
\end{itemize}
\end{lemma}

\section{A variable-time energy principle for dissipative semigroups}
\label{sec:euclidean-variable}
In this section, we establish a variable-time estimate for strongly
continuous semigroups on $L^2$.
\begin{definition}
Let $(\mathsf X,\mathcal B,\mathfrak m)$ be a sigma-finite measure
space. A family $\{T_t\}_{t\in[0,\infty)}$ of bounded linear operators on
$L^2(\mathsf X,\mathfrak m)$ is called a \emph{strongly continuous
semigroup} if
$T_0=I$,
$T_{s+t}=T_sT_t$ for every $s,t\in[0,\infty)$,
and
\begin{align*}
\lim_{t\to0^+}
\|T_tf-f\|_{L^2\left(\mathsf X,\mathfrak m\right)}
=0
\end{align*}
for every $f\in L^2(\mathsf X,\mathfrak m)$.
We define
\begin{align*}
\operatorname{Dom}\left(A\right)
:=
\left\{
f\in L^2\left(\mathsf X,\mathfrak m\right):
\lim_{t\to0^+}\frac{f-T_tf}{t}
\text{ exists in }L^2\left(\mathsf X,\mathfrak m\right)
\right\},
\end{align*}
and, for any $f\in\operatorname{Dom}(A)$,
\begin{align*}
Af
:=
\lim_{t\to0^+}\frac{f-T_tf}{t}.
\end{align*}
Then $A$ is a linear operator on $\rm{Dom}(A)$,
called the \emph{negative infinitesimal generator} of the semigroup $\{T_t\}_{t\in[0,\infty)}$.
\end{definition}
For every $g\in\operatorname{Dom}(A)$ and $t\in[0,\infty)$, we have
$T_tg\in\operatorname{Dom}(A)$,
$AT_tg=T_tAg$,
and
\begin{align}
\label{eq:semigroup-fundamental-identity}
g-T_tg
=
\int_0^tT_s\left(Ag\right)\,\mathrm d s
\end{align}
in $L^2(\mathsf X,\mathfrak m)$.
Let $g\in\operatorname{Dom}(A)$ and choose a measurable representative of $g$,
still denoted by $g$.
Since $s\mapsto T_s(Ag)$ is continuous as an
$L^2(\mathsf X,\mathfrak m)$-valued map, it follows that,
for every $R\in(0,\infty)$,
\begin{align*}
\int_0^R\|T_s\left(Ag\right)\|_{L^2\left(\mathsf X,\mathfrak m\right)}^2\,\mathrm d s<\infty.
\end{align*}
Hence, using \begin{align*}
L^2\left(\left(0,R\right);L^2\left(\mathsf X,\mathfrak m\right)\right)
\cong
L^2\left(\left(0,R\right)\times\mathsf X,\,\mathrm d s\otimes\,\mathrm d\mathfrak m\right),
\end{align*}
where $\cong$ denotes the natural isometric isomorphism between the two spaces,
we may choose a jointly measurable function
$H_g:(0,\infty)\times\mathsf X\to\mathbb C$
such that
\begin{align}\label{2155}
H_g\left(s,\cdot\right)=T_s\left(Ag\right)
\end{align}
in $L^2(\mathsf X,\mathfrak m)$
for almost every $s\in(0,\infty)$.
Using Minkowski's integral inequality, we find that, for every $R\in(0,\infty)$,
\begin{align*}
\left\|
\int_0^R |H_g\left(s,\cdot\right)|\,\mathrm d s
\right\|_{L^2\left(\mathsf X,\mathfrak m\right)}
\le
\int_0^R
\|T_s\left(Ag\right)\|_{L^2\left(\mathsf X,\mathfrak m\right)}
\,\mathrm d s
<\infty.
\end{align*}
Taking the union of the exceptional null sets for positive integer
$R$, we obtain a measurable set
$N_g\subset\mathsf X$ with $\mathfrak m(N_g)=0$ such that,
for every $x\in\mathsf X\setminus N_g$ and every $R\in(0,\infty)$,
\begin{align*}
\int_0^R |H_g\left(s,x\right)|\,\mathrm d s<\infty.
\end{align*}
For any $t\in[0,\infty)$, we define
\begin{align*}
\widetilde T_tg\left(x\right)
:=
\begin{cases}
\displaystyle
g\left(x\right)-\int_0^tH_g\left(s,x\right)\,\mathrm d s,
&x\in\mathsf X\setminus N_g,\\[2mm]
0,
&x\in N_g.
\end{cases}
\end{align*}
Then the map
$(t,x)\longmapsto \widetilde T_tg(x)$
is measurable. Moreover, from
\eqref{eq:semigroup-fundamental-identity},
we infer that, for every $t\in[0,\infty)$,
$\widetilde T_tg=T_tg$
in $L^2(\mathsf X,\mathfrak m)$.
Here, and thereafter, we still write
$T_tg$ instead of $\widetilde T_tg$. Let
$\vartheta:\mathsf X\to(0,\infty]$
be measurable and let
$\mathsf X_\vartheta
:=
\{x\in\mathsf X:\vartheta(x)\in(0,\infty)\}.$
We define
\begin{align*}
D_{\vartheta,g}\left(x\right)
:=
\begin{cases}
\displaystyle
g\left(x\right)-T_{\vartheta\left(x\right)}g\left(x\right)
=
\int_0^{\vartheta\left(x\right)}H_g\left(s,x\right)\,\mathrm d s,
&x\in\mathsf X_\vartheta\setminus N_g,\\[3mm]
0,
&x\notin\mathsf X_\vartheta
\text{ or }x\in N_g.
\end{cases}
\end{align*}
In particular, $D_{\vartheta,g}$ is measurable.

\begin{theorem}
\label{thm:variable-time-graph-square}
Let $\{T_t\}_{t\in[0,\infty)}$ be a strongly continuous semigroup on
$L^2(\mathsf X,\mathfrak m)$ and let $-A$ be its infinitesimal 
generator in the usual sign convention. Let $g\in\operatorname{Dom}(A)$ and
$\vartheta:\mathsf X\to(0,\infty]$ be a measurable
function.
Then
\begin{align*}
\int_{\mathsf X_\vartheta}
\frac{|D_{\vartheta,g}(x)|^2}{\vartheta(x)}
\,\mathrm d\mathfrak m\left(x\right)
\le
\int_0^\infty
\left\|T_s\left(Ag\right)\right\|_{L^2\left(\mathsf X,\mathfrak m\right)}^2\,\mathrm d s.
\end{align*}
Further assume that $A$ is nonnegative and self-adjoint so that
$T_t=e^{-tA}$. Then
\begin{align*}
\int_{\mathsf X_\vartheta}
\frac{|D_{\vartheta,g}(x)|^2}{\vartheta(x)}
\,\mathrm d\mathfrak m\left(x\right)
\le
\frac12\left\|A^{\frac12}g\right\|_{L^2\left(\mathsf X,\mathfrak m\right)}^2.
\end{align*}
\end{theorem}

\begin{proof}
Using the definition of
$D_{\vartheta,g}$ and the Cauchy--Schwarz inequality, we conclude that,
for any $x\in\mathsf X_\vartheta\setminus N_g$,
\begin{align*}
\frac{|D_{\vartheta,g}(x)|^2}{\vartheta(x)}=
\frac{1}{\vartheta(x)}
\left|
\int_0^{\vartheta\left(x\right)}H_g\left(s,x\right)\,\mathrm d s
\right|^2\le
\int_0^{\vartheta\left(x\right)}|H_g\left(s,x\right)|^2\,\mathrm d s.
\end{align*}
From $\mathfrak m(N_g)=0$ and Tonelli's theorem, we infer that
\begin{align*}
\int_{\mathsf X_\vartheta}
\frac{|D_{\vartheta,g}(x)|^2}{\vartheta(x)}
\,\mathrm d\mathfrak m\left(x\right)
&\le
\int_{\mathsf X_\vartheta}
\int_0^{\vartheta\left(x\right)}
|H_g\left(s,x\right)|^2\,\mathrm d s\,\mathrm d\mathfrak m\left(x\right)\\
&\le
\int_0^\infty
\int_{\mathsf X}
|H_g\left(s,x\right)|^2
\,\mathrm d\mathfrak m\left(x\right)\,\mathrm d s\\
&=
\int_0^\infty
\|T_s\left(Ag\right)\|_{L^2\left(\mathsf X,\mathfrak m\right)}^2\,\mathrm d s.
\end{align*}

Now, suppose that $A$ is nonnegative and self-adjoint.
Let $E_A$ be the spectral resolution of $A$, and let
$\nu_g(B):=\langle E_A(B)g,g\rangle$
for every Borel set $B\subset[0,\infty)$. By the fact
that $A$ is nonnegative and self-adjoint,
the spectral theorem (see, for instance,
\cite[Theorem~VIII.6]{ReedSimonI}), and Tonelli's theorem, we find that
\begin{align*}
\int_0^\infty\|T_s\left(Ag\right)\|_{L^2\left(\mathsf X,\mathfrak m\right)}^2\,\mathrm d s
&=\int_{[0,\infty)}
\left[\int_0^\infty\lambda^2e^{-2s\lambda}\,\mathrm d s\right]
\,\mathrm d\nu_g\left(\lambda\right)\\
&=
\frac12
\int_{0}^\infty
\lambda\,\mathrm d\nu_g\left(\lambda\right)=
\frac12\left\|A^{\frac12}g\right\|_{L^2\left(\mathsf X,\mathfrak m\right)}^2.
\end{align*}
The bracketed integral is zero when $\lambda=0$ and equals
$\frac{\lambda}{2}$ when $\lambda\in(0,\infty)$.
This completes the proof of Theorem \ref{thm:variable-time-graph-square}.
\end{proof}

\section{Brownian absorption and finite-time capacitary killing}
\label{sec:brownian-exit}

In this section, we introduce the killed Brownian semigroup and the
Brownian exit distance, and establish the Hardy inequality associated
with it. We then derive a finite-time exit estimate from a
capacitary condition and use it to prove Theorem~\ref{thm:main}(i).

\subsection{Killed Brownian motion and absorption quantiles}

We first recall some notation and standard facts concerning Brownian
motion and the killed Brownian semigroup.
Let $(W_t)_{t\in[0,\infty)}$ be the standard $n$-dimensional \emph{Brownian motion} and,
for any $x\in\mathbb R^n$ and $t\in[0,\infty)$, let
$X_t^x:=x+\sqrt2\,W_t.$
When there is no ambiguity, we simply write $X_t$ instead of $X_t^x$.
We write $\mathbb P_x$ and $\mathbb E_x$, respectively, for probability and expectation when
the process starts from $x$. Thus, for any $t\in(0,\infty)$,
\begin{align}\label{1500}
\mathbb P_x\left(X_t\in\,\mathrm d y\right)
=\left(4\pi t\right)^{-\frac{n}{2}}
\exp\left(-\frac{|x-y|^2}{4t}\right)\,\mathrm d y.
\end{align}
We denote by
$\{\mathcal F_t\}_{t\in[0,\infty)}$ its completed right-continuous natural
\emph{filtration}.
Let $\mathcal W:=C([0,\infty),\mathbb R^n)$
equipped with the compact-open topology and its Borel sigma-algebra.
For any closed set $E\subset\mathbb R^n$, any open set $U\subset\mathbb R^n$, and
any $w\in\mathcal W$, we define
\begin{align*}
T_E\left(w\right):=\inf\{t\in[0,\infty):w\left(t\right)\in E\}
\quad\text{and}\quad
\tau_U\left(w\right):=\inf\{t\in[0,\infty):w\left(t\right)\notin U\},
\end{align*}
with $\inf\emptyset:=\infty$.  The maps $T_E,\tau_U:\mathcal W\to[0,\infty]$ are Borel measurable.
We regard $X=(X_t)_{t\in[0,\infty)}$ as a $\mathcal W$-valued random element;
thus, $X$ denotes the whole path, whereas $X_t\in\mathbb R^n$ denotes its
position at time $t$.

\begin{definition}

For any open set $U\subset\mathbb R^n$, any bounded Borel function
$h:U\to\mathbb R$ extended by zero outside $U$, $x\in U$, and
$t\in[0,\infty)$, we define the
\emph{killed Brownian semigroup} by
\begin{align*}
\left(P_t^Uh\right)\left(x\right)
:=
\mathbb E_x\!\left[
h\left(X_t\right)\mathbf 1_{\{t<\tau_U\}}
\right].
\end{align*}
In particular,
\begin{align*}
P_t^U\mathbf 1\left(x\right)=\mathbb P_x\left(t<\tau_U\right)
\quad \mathrm{and} \quad
1-P_t^U\mathbf 1\left(x\right)=\mathbb P_x\left(\tau_U\le t\right).
\end{align*}
\end{definition}

Consider the \emph{closed symmetric form $\mathcal E_U$} on $L^2(U)$
with domain $H_0^1(U)$, defined by
\begin{align*}
\mathcal E_U\left(u,v\right)
:=
\int_U\nabla u\left(x\right)\cdot\overline{\nabla v\left(x\right)}\,\mathrm d x.
\end{align*}
For real functions the conjugation has no effect. We use the same
notation for the complex extension, whose form and inner product
are sesquilinear and linear in the first argument. The associated
operator and semigroup extend complex linearly.
Let $A_U$ be the nonnegative self-adjoint operator
associated with $\mathcal E_U$, i.e.
\begin{align*}
\operatorname{Dom}\left(A_U\right):=
\left\{u\in H_0^1\left(U\right):
\begin{array}{l}
\text{there exists }f\in L^2\left(U\right)\text{ such that}\\
\mathcal E_U\left(u,v\right)=\langle f,v\rangle_{L^2\left(U\right)}
\text{ for every }v\in H_0^1\left(U\right)
\end{array}
\right\}
\end{align*}
and $A_Uu:=f$. By the representation theorem (see \cite[Theorem~1.3.1]{FukushimaOshimaTakeda}),
we find that
$\operatorname{Dom}(A_U^{\frac12})=H_0^1(U)$
and, for any $u\in H_0^1(U)$,
\begin{align}\label{810}
\left\|A_U^{\frac12}u\right\|_{L^2\left(U\right)}^2
=
\mathcal E_U\left(u,u\right)
=
\int_U|\nabla u\left(x\right)|^2\,\mathrm d x.
\end{align}
Moreover, $C_{\mathrm{c}}^\infty(U)\subset\operatorname{Dom}(A_U)$ and
$A_Ug=-\Delta g$ for every $g\in C_{\mathrm{c}}^\infty(U)$.
The operators $\{P_t^U\}_{t\in[0,\infty)}$ extend to a strongly continuous
symmetric sub-Markovian contraction semigroup on $L^2(U)$ and,
for any $t\in[0,\infty)$,
$ P_t^U=e^{-tA_U};$
see, for instance,
\cite[Section~2.5]{ChungZhao}.
For every bounded Borel function $h$ on $U$, the map
$(t,x)\longmapsto P_t^Uh(x)$
is Borel measurable on $[0,\infty)\times U$.  Moreover, for any
$g\in C_{\rm{c}}^\infty(U)$, $x\in U$, and $t\in[0,\infty)$, we have
the pointwise Dynkin formula
\begin{align}\label{2156}
g\left(x\right)-P_t^Ug\left(x\right)
=
\int_0^tP_s^U\left(A_Ug\right)\left(x\right)\,\mathrm d s.
\end{align}

The preceding variable-time estimate naturally leads to a probabilistic
distance. Related probabilistic exit-time scales have previously been used in
spectral estimates; see, for example,
\cite{LuSteinerbergerDV}.

\begin{definition}
\label{def:brownian-exit-distance}
Let $\Omega\subset\mathbb R^n$ be an open set and $p\in(0,1)$.
For $x\in\Omega$, define the \emph{$p$-exit distance} 
(equivalently, the absorption-quantile length) by
\begin{align*}
\rho_p\left(x\right)
:=\inf\left\{
r\in\left(0,\infty\right):
\mathbb P_x\left(\tau_\Omega\le r^2\right)\ge p
\right\}.
\end{align*}
\end{definition}

\begin{lemma}
\label{lem:brownian-exit-distance-measurable}
Let $\Omega\subset\mathbb R^n$ be open.
For every $p\in(0,1)$, the function
$\rho_p:\Omega\to(0,\infty]$ is Borel measurable. Moreover, if $\rho_p(x)\in(0,\infty)$, then
$\mathbb P_x(\tau_\Omega\le\rho_p(x)^2)\ge p.$
\end{lemma}

\begin{proof}
We first claim that $\rho_p(x)\in(0,\infty]$ for every $x\in\Omega$.
If $\Omega=\mathbb R^n$, then $\rho_p(x)=\infty$. Otherwise, by the continuity of Brownian paths,
we have
\begin{align*}
\mathbb P_x\left(\tau_\Omega\le r^2\right)
\le
\mathbb P_x\left(
\sup_{t\in\left[0,r^2\right]}|X_t-x|\ge\operatorname{dist}\left(x,\Omega^\complement\right)
\right)\to0
\end{align*}
as $r\to0^+$, which further implies that
$\rho_p(x)\in(0,\infty]$.
We claim that, for any $a\in(0,\infty)$,
\begin{align}\label{1430}
\left\{x\in\Omega:\rho_p\left(x\right)\in\left(0,a\right)\right\}
=
\bigcup_{r\in\left(0,a\right)\cap\mathbb Q}
\left\{x\in\Omega:1-P_{r^2}^\Omega\mathbf 1\left(x\right)\ge p\right\},
\end{align}
where $\mathbb Q$ denotes the set of all rational numbers.
Indeed, if $\rho_p(x)\in(0,a)$, then there exists
$r\in(\rho_p(x),a)\cap\mathbb Q$.
By the definition of
$\rho_p(x)$ and the monotonicity of
$s\mapsto 1-P_{s^2}^\Omega\mathbf 1(x)$, we have
$1-P_{r^2}^\Omega\mathbf 1(x)\ge p$. Conversely, if
$1-P_{r^2}^\Omega\mathbf 1(x)\ge p$ for some $r\in(0,a)\cap\mathbb Q$, then
$\rho_p(x)\in(0,r]\subset(0,a)$.
Then the above claim \eqref{1430} holds and
the Borel measurability follows from the Borel measurability of
$x\mapsto P_{r^2}^\Omega\mathbf 1(x)$. If $\rho_p(x)\in(0,\infty)$, then
there exists a sequence $\{r_j\}_{j\in\mathbb N}$ such that
$r_j\to\rho_p(x)^+$ and, for any $j\in\mathbb N$,
$\mathbb P_x(\tau_\Omega\le r_j^2)\ge p$.
Since the distribution function
$t\mapsto\mathbb P_x(\tau_\Omega\le t)$ is right-continuous, it follows that
\begin{align*}
\mathbb P_x\left(\tau_\Omega\le\rho_p\left(x\right)^2\right)
=
\lim_{j\to\infty}\mathbb P_x\left(\tau_\Omega\le r_j^2\right)
\ge p,
\end{align*}
which completes the proof of Lemma \ref{lem:brownian-exit-distance-measurable}.
\end{proof}

\subsection{The local absorption-time energy inequality}

\begin{proof}[Proof of Theorem~\ref{thm:exit-distance-hardy}]
We first claim that, for every compact set $K\Subset\Omega$,
\begin{align}
\label{eq:exit-distance-capacity-bound}
\int_K\rho_p\left(x\right)^{-2}\,\mathrm d x
\le
\frac{1}{2p^2}\operatorname{cap}_\Omega\left(K\right).
\end{align}
To prove this, fix $\varepsilon\in(0,\infty)$ and
choose $\delta\in(0,1)$ sufficiently small such that
\begin{align}\label{1434}
\left(1+\delta\right)^2\left[\operatorname{cap}_\Omega\left(K\right)+\delta\right]
\le \operatorname{cap}_\Omega\left(K\right)+\varepsilon.
\end{align}
By the definition of $\operatorname{cap}_\Omega(K)$, there exists
$\varphi\in C_{\rm{c}}^\infty(\Omega)$ such that $\varphi\ge1$ on $K$ and
\begin{align}\label{1435}
\int_\Omega|\nabla\varphi\left(x\right)|^2\,\mathrm d x
\le \operatorname{cap}_\Omega\left(K\right)+\delta.
\end{align}
Choose $\chi\in C^\infty(\mathbb R)$ such that
$0\le\chi\le1$, $\chi=0$ on $(-\infty,0]$, $\chi=1$ on $[1,\infty)$,
and $|\chi'|\le1+\delta$, and let $g:=\chi\circ\varphi$. Then
$g\in C_{\rm{c}}^\infty(\Omega)$, $g=1$ on $K$, and $0\le g\le1$.
From the chain rule, the bound $|\chi'|\le1+\delta$, \eqref{1435}, and \eqref{1434}, we deduce that
\begin{align}
\label{eq:exit-proof-g-energy}
\int_\Omega|\nabla g\left(x\right)|^2\,\mathrm d x
&\le \left(1+\delta\right)^2\int_\Omega|\nabla\varphi\left(x\right)|^2\,\mathrm d x\notag\\
&\le \left(1+\delta\right)^2\left[\operatorname{cap}_\Omega\left(K\right)+\delta\right]
\le \operatorname{cap}_\Omega\left(K\right)+\varepsilon.
\end{align}
For any $x\in\Omega$, let
\begin{align*}
\vartheta_p\left(x\right):=
\begin{cases}
\rho_p\left(x\right)^2,&\rho_p\left(x\right)\in\left(0,\infty\right),\\
\infty,&\rho_p\left(x\right)=\infty.
\end{cases}
\end{align*}
By Lemma~\ref{lem:brownian-exit-distance-measurable}, we find that
$\vartheta_p$ is Borel measurable.  From $g=1$ on $K$, $0\le g\le1$,
the monotonicity of the killed Brownian semigroup, and
Lemma \ref{lem:brownian-exit-distance-measurable}, we deduce that, for
$x\in K$ with $\rho_p(x)\in(0,\infty)$,
\begin{align}\label{2159}
g\left(x\right)-P_{\vartheta_p\left(x\right)}^\Omega g\left(x\right)
=1-P_{\vartheta_p\left(x\right)}^\Omega g\left(x\right)
\ge1-P_{\vartheta_p\left(x\right)}^\Omega\mathbf 1\left(x\right)
=\mathbb P_x\left(\tau_\Omega\le\rho_p\left(x\right)^2\right)\ge p.
\end{align}
We choose $H_g(s,x):=P_s^\Omega(A_\Omega g)(x)$ satisfying
\eqref{2155}. Since
$|H_g(s,x)|\le\|A_\Omega g\|_{L^\infty(\Omega)}$, the
exceptional set $N_g$ may be taken to be empty. Formula~\eqref{2156}
then gives $D_{\vartheta_p,g}(x)=g(x)-P_{\vartheta_p(x)}^\Omega g(x)$
when $\vartheta_p(x)\in(0,\infty)$, and $D_{\vartheta_p,g}(x)=0$ otherwise.
Writing $\Omega_{\vartheta_p}:=\{x\in\Omega:\vartheta_p(x)\in(0,\infty)\}$,
and using \eqref{2159}, Theorem~\ref{thm:variable-time-graph-square},
and \eqref{810},
we conclude that
\begin{align*}
p^2\int_K\rho_p\left(x\right)^{-2}\,\mathrm d x
\le
\int_{K\cap\Omega_{\vartheta_p}}
\frac{|D_{\vartheta_p,g}(x)|^2}{\vartheta_p(x)}
\,\mathrm d x
\le\frac12\int_\Omega|\nabla g\left(x\right)|^2\,\mathrm d x.
\end{align*}
Using \eqref{eq:exit-proof-g-energy} and letting
$\varepsilon\to0^+$, we find that
\eqref{eq:exit-distance-capacity-bound} holds.

Define the Borel measure $\,\mathrm d\mu_p(x):=\rho_p(x)^{-2}\,\mathrm d x$.
The compact-set bound \eqref{eq:exit-distance-capacity-bound} shows
that it is locally finite, since $\operatorname{cap}_\Omega(K)\in[0,\infty)$
for every compact $K\Subset\Omega$.  From
\eqref{eq:exit-distance-capacity-bound}, we infer that, for every compact
$K\Subset\Omega$,
\begin{align*}
\mu_p\left(K\right)\le\frac{1}{2p^2}\operatorname{cap}_\Omega\left(K\right).
\end{align*}
Applying Lemma~\ref{thm:isocap}\textup{(ii)}, we conclude that, for every
$u\in C_{\rm{c}}^\infty(\Omega)$,
\begin{align*}
\int_\Omega\frac{|u(x)|^2}{\rho_p(x)^2}\,\mathrm d x
\le
\frac{2}{p^2}\int_\Omega|\nabla u\left(x\right)|^2\,\mathrm d x.
\end{align*}
This completes the proof of Theorem~\ref{thm:exit-distance-hardy}.
\end{proof}

\subsection{Green equilibrium and the control of late hitting}

Let $D:=B(c,R)$ and let $G_D$ denote the Dirichlet Green function
of $-\Delta$ in $D$. For $x\in D\setminus\{c\}$, let
\begin{align*}
x^*:=c+\frac{R^2(x-c)}{|x-c|^2}.
\end{align*}
Then, for $x,z\in D$ with $x\ne z$,
\begin{align*}
G_D\left(x,z\right)
=
\frac1{\kappa_n}
\left[
|z-x|^{2-n}
-
\left(\frac{R}{|x-c|}\right)^{n-2}
|z-x^*|^{2-n}
\right]
\end{align*}
with the usual limiting interpretation when $x=c$; in particular,
\begin{align}\label{1514}
G_D\left(c,z\right)
=
\frac1{\kappa_n}
\left(|z-c|^{2-n}-R^{2-n}\right);
\end{align}
see, for instance, \cite[Section~2.2.4]{EvansPDE}.

The following lemma is a consequence of classical Green
equilibrium theory; see, for instance,
\cite[Theorem~5.1, p.~190]{PortStone} and \cite[Chapter~6]{PortStone} and
see also \cite[Sections~2.4--2.5]{BetsakosBoudabraMarkowsky}.

\begin{lemma}
\label{lem:self-contained-ball-equilibrium}
Let $D:=B(c,R)$ be a ball, let $G_D$ denote the Dirichlet Green
function of $-\Delta$ in $D$, and let $K\Subset D$ be compact.  For
$y\in D$, we define
\begin{align}\label{hED}
h_K^D\left(y\right):=\mathbb P_y\left(T_K<\tau_D\right).
\end{align}
Then there exists a finite nonnegative Radon measure $\mu_K^D$ supported
on $K$ such that
\begin{align}
\label{eq:compact-capacity-mass-energy}
\mu_K^D\left(K\right)=\operatorname{cap}_D\left(K\right).
\end{align}
Moreover,
\begin{align}
\label{eq:compact-green-representation}
h_K^D\left(y\right)
=
\int_KG_D\left(y,z\right)\,\mathrm d\mu_K^D\left(z\right)
\end{align}
for almost every $y\in D$ and, in particular, for every
$y\in D\setminus K$.
\end{lemma}

\begin{lemma}
\label{lem:compact-eventual-hit}
Let $\Omega\subset\mathbb R^n$ be an open set and $x_0\in\Omega$.
Write $F:=\mathbb R^n\setminus\Omega$ and let
$\kappa_n:=n(n-2)|B(\mathbf0,1)|$ and $a_n:=1-2^{2-n}$.
Let $r\in(0,\infty)$, $\beta\in(0,1]$, and
$K\subset F\cap\overline{B(x_0,r)}$ be a compact set. Assume that
\begin{align*}
\operatorname{cap}\left(K\right)\ge\beta\operatorname{cap}\left(B\left(\mathbf 0,r\right)\right)=\beta\kappa_nr^{n-2}.
\end{align*}
Then
\begin{align*}
\mathbb P_{x_0}\left(T_K<\tau_{B\left(x_0,2r\right)}\right)\ge a_n\beta.
\end{align*}
\end{lemma}

\begin{proof}
Using $x_0\in\Omega$ and $K\subset F$, we obtain $x_0\notin K$.
Let $D:=B(x_0,2r)$ and $h^D_K$ be the same as in \eqref{hED}. From \eqref{1514}, it follows that,
for any
$z\in\overline{B(x_0,r)}\setminus\{x_0\}$,
\begin{align}
\label{eq:green-lower}
G_D\left(x_0,z\right)
\ge
\frac{a_n}{\kappa_n}r^{2-n}.
\end{align}
Let $\mu_K^D$ be the same as in
Lemma~\ref{lem:self-contained-ball-equilibrium}.  Using
\eqref{eq:compact-green-representation},
\eqref{eq:green-lower},
and \eqref{eq:compact-capacity-mass-energy},
we find that
\begin{align}\label{1713}
h_K^D\left(x_0\right)
=\int_KG_D\left(x_0,z\right)\,d \mu_K^D\left(z\right)
\ge\frac{a_n}{\kappa_n}r^{2-n}\mu_K^D\left(K\right)
=\frac{a_n}{\kappa_n}r^{2-n}\operatorname{cap}_D\left(K\right),
\end{align}
which, together with Lemma~\ref{lem:capacity-package}(iv),
further implies that
\begin{align*}
h_K^D\left(x_0\right)
\ge\frac{a_n}{\kappa_n}r^{2-n}\operatorname{cap}\left(K\right)
\ge a_n\beta.
\end{align*}
This completes the proof of Lemma~\ref{lem:compact-eventual-hit}.
\end{proof}

\begin{lemma}
\label{lem:late-hit}
Let $n\in[3,\infty)\cap\mathbb N$, and write
$\kappa_n:=n(n-2)|B(\mathbf0,1)|$ and $a_n:=1-2^{2-n}$.
Let $x_0\in\mathbb R^n$, $r\in(0,\infty)$,
$K\subset\overline{B(x_0,r)}$ be compact with $x_0\notin K$,
$D:=B(x_0,2r)$, and $h^D_K$ be the same as in \eqref{hED}.
Then, for any $T\in(0,\infty)$,
\begin{align}
\label{eq:late-hit-bound}
\mathbb P_{x_0}\left(T_K\in\left(Tr^2,\tau_D\right)\right)
\le\frac{2\kappa_n}{na_n}\left(4\pi T\right)^{-\frac n2}h^D_K\left(x_0\right).
\end{align}
In particular, if $T_n:=\frac1{4\pi}
(\frac{4\kappa_n}{na_n})^{\frac2n},$
then
\begin{align*}
\mathbb P_{x_0}\left(T_K\in\left(T_nr^2,\tau_D\right)\right)\le\frac12h^D_K\left(x_0\right).
\end{align*}
\end{lemma}

\begin{proof}
Fix $T\in(0,\infty)$ and let $t:=Tr^2$.
Let $A_t:=\{t<T_K\wedge\tau_D\}.$
Then $A_t\in\mathcal F_t$.  On $A_t$, the Brownian path has reached
neither $K$ nor $\mathbb R^n\setminus D$ by time $t$, and hence
$X_t\in D\setminus K$.  By the Markov property of Brownian motion at
time $t$ (see, for instance,
\cite[Theorem~6.15]{KaratzasShreve}), we find that
\begin{align*}
\mathbb E_{x_0}\left[
\mathbf 1_{\{t<T_K<\tau_D\}}\mid\mathcal F_t
\right]=\mathbb P_{x_0}
\left(t<T_K<\tau_D\mid\mathcal F_t\right)
=
\mathbf 1_{A_t}h_K^D\left(X_t\right)
\end{align*}
almost surely.
Taking expectations and using the tower property of conditional expectation,
we conclude that
\begin{align*}
\mathbb P_{x_0}\left(t<T_K<\tau_D\right)
=
\mathbb E_{x_0}\left[
\mathbf 1_{\{t<T_K\wedge\tau_D\}}h_K^D\left(X_t\right)
\right].
\end{align*}
Since $h_K^D\ge0$ and
$\{t<T_K\wedge\tau_D\}\subset\{t<\tau_D\}$, it follows that
\begin{align}\label{1716}
\mathbb P_{x_0}\left(T_K\in\left(t,\tau_D\right)\right)
\le
\mathbb E_{x_0}\!\left[
\mathbf 1_{\{t<\tau_D\}}h_K^D\left(X_t\right)
\right]=P_t^Dh_K^D\left(x_0\right).
\end{align}
Using \eqref{1500} and
$\{t<\tau_D\}\subset\{X_t\in D\}$, we obtain
\begin{align}\label{1717}
P_t^Dh_K^D\left(x_0\right)
&=
\mathbb E_{x_0}\!\left[
h_K^D\left(X_t\right)\mathbf 1_{\{t<\tau_D\}}
\right]\le
\mathbb E_{x_0}\!\left[
h_K^D\left(X_t\right)\mathbf 1_{\{X_t\in D\}}
\right]\nonumber\\
&=\int_D
\left(4\pi t\right)^{-\frac n2}
\exp\!\left(-\frac{|x_0-y|^2}{4t}\right)
h_K^D\left(y\right)\,\mathrm d y\nonumber\\
&\le
\left(4\pi t\right)^{-\frac n2}
\int_Dh_K^D\left(y\right)\,\mathrm d y.
\end{align}
Let $\mu_K^D$ be the same as in
Lemma~\ref{lem:self-contained-ball-equilibrium}. From \eqref{eq:compact-green-representation}
and Tonelli's theorem, we infer that
\begin{align}\label{1707}
\int_Dh_K^D\left(y\right)\,\mathrm d y
&=
\int_K
\left[
\int_DG_D\left(y,z\right)\,\mathrm d y
\right]
\,\mathrm d\mu_K^D\left(z\right).
\end{align}
By the definition of $G_D$ and $D=B(x_0,2r)$, we find that, for any $z\in D$,
\begin{align*}
\int_DG_D\left(y,z\right)\,\mathrm d y
=
\frac{4r^2-|z-x_0|^2}{2n}
\le
\frac{2r^2}{n}.
\end{align*}
Combining this, \eqref{1707}, and
\eqref{eq:compact-capacity-mass-energy},
we obtain
\begin{align*}
\int_Dh_K^D\left(y\right)\,\mathrm d y
&\le
\frac{2r^2}{n}\mu_K^D\left(K\right)
=
\frac{2r^2}{n}\operatorname{cap}_D\left(K\right),
\end{align*}
which, together with \eqref{1713}, further implies that
\begin{align*}
\int_Dh_K^D\left(y\right)\,\mathrm d y
\le
\frac{2\kappa_n}{na_n}r^nh_K^D\left(x_0\right).
\end{align*}
From this, \eqref{1716}, \eqref{1717}, and $t=Tr^2$, we infer that
\begin{align*}
\mathbb P_{x_0}\left(T_K\in\left(Tr^2,\tau_D\right)\right)
\le
\frac{2\kappa_n}{na_n}
\left(4\pi T\right)^{-\frac n2}
h_K^D\left(x_0\right),
\end{align*}
which is \eqref{eq:late-hit-bound}.
This completes the proof of Lemma~\ref{lem:late-hit}.
\end{proof}

\subsection{Finite-time killing and comparison of trapping scales}

\begin{proof}[Proof of Theorem~\ref{thm:compact-killing}]
Use $x_0=x$ in the notation of the preceding lemmas and let
$D:=B(x,2r)$. By
Lemmas~\ref{lem:compact-eventual-hit} and~\ref{lem:late-hit}, we find that
\begin{align*}
\mathbb P_{x_0}\left(T_K\le T_nr^2\right)
&\ge \mathbb P_{x_0}\left(T_K<\tau_D\right)
-\mathbb P_{x_0}\left(T_K\in\left(T_nr^2,\tau_D\right)\right)\\
&\ge\frac12\mathbb P_{x_0}\left(T_K<\tau_D\right)
\ge\frac{a_n\beta}{2}.
\end{align*}
Since $K\subset F:=\mathbb R^n\setminus\Omega$, hitting $K$ implies that
the path has exited $\Omega$.  This proves \eqref{eq:compact-killing}, and
\eqref{eq:compact-survival} follows from
$P_t^\Omega\mathbf1(x_0)=\mathbb P_{x_0}(t<\tau_\Omega)$,
which completes the proof of Theorem~\ref{thm:compact-killing}.
\end{proof}

\begin{corollary}
\label{cor:scale-comparison}
Let $n\in[3,\infty)\cap\mathbb N$ and
$\Omega\subset\mathbb R^n$ be open.
Let $a_n$ and $T_n$ be the dimensional
constants in \eqref{eq:intro-constants}.
For $\alpha\in(0,1]$, let $p_\alpha=\frac{a_n\alpha}{2}$. Then, for any
$x\in\Omega$,
\begin{align}\label{eq:scale-comparison}
\rho_{p_\alpha}\left(x\right)\le\sqrt{T_n}\,d_\alpha\left(x\right).
\end{align}
If $d_\alpha(x)\in(0,\infty)$, then also
\begin{align}\label{eq:killing-at-trapping-scale}
\mathbb P_x\left(\tau_\Omega\le T_nd_\alpha\left(x\right)^2\right)\ge p_\alpha.
\end{align}
\end{corollary}

\begin{proof}
When $d_\alpha(x)=\infty$, the comparison is automatic.
Otherwise choose admissible radii $r_j\to d_\alpha(x)$ with
$r_j\ge d_\alpha(x)$. The compact sets
$K_j=\overline{F\cap B(x,r_j)}\subset F\cap\overline{B(x,r_j)}$
satisfy the hypotheses of Theorem~\ref{thm:compact-killing} with
$\beta=\alpha$. Hence,
$\mathbb P_x(\tau_\Omega\le T_nr_j^2)\ge p_\alpha.$
Right continuity of the exit-time distribution, approached from
above, proves \eqref{eq:killing-at-trapping-scale}.
Definition~\ref{def:brownian-exit-distance} then proves
\eqref{eq:scale-comparison}.
\end{proof}

\begin{proof}[Proof of Theorem~\ref{thm:main}(i)]
Corollary~\ref{cor:scale-comparison} gives the pointwise weight bound
$d_\alpha^{-2}\le T_n\rho_{p_\alpha}^{-2}$, including points at
which $d_\alpha=\infty$. Therefore
\begin{align*}
\int_\Omega\frac{|u\left(x\right)|^2}{d_\alpha\left(x\right)^2}\,\mathrm d x
\le T_n\int_\Omega\frac{|u\left(x\right)|^2}{\rho_{p_\alpha}\left(x\right)^2}\,\mathrm d x
\le\frac{2T_n}{p_\alpha^2}\int_\Omega|\nabla u\left(x\right)|^2\,\mathrm d x
=\frac{C_n}{\alpha^2}\int_\Omega|\nabla u\left(x\right)|^2\,\mathrm d x.
\end{align*}
This proves the stated result. If one extends the definition to
$\alpha\in(1,\infty)$, then $d_\alpha\equiv\infty$ by
Lemma~\ref{lem:dalpha-elementary-properties}(iv), and the
corresponding inequality is trivial.
\end{proof}

\section{Codimension-two absorbers and the sharp loading law}
\label{sec:alpha-sharpness}

In this section, we prove Theorem \ref{thm:main}(ii).
We first record the elementary capacity estimate used in the construction.

\begin{lemma}
\label{lem:codim-two-cylinder-capacity}
For $R\in[2,\infty)$, let
\begin{align*}
\mathcal C_R
:=[-R,R]^{n-2}\times\overline{B_{\mathbb R^2}\left(\mathbf 0,1\right)}
\subset\mathbb R^{n-2}\times\mathbb R^2=\mathbb R^n.
\end{align*}
Then there exists a positive constant $c_{\mathrm{cyl},n}$ such that
$\operatorname{cap}(\mathcal C_R)
\ge \frac{c_{\mathrm{cyl},n}R^{n-2}}{\log R}.$
\end{lemma}
\begin{proof}
Using the energy characterization of Newtonian capacity
(see, for instance, \cite[Section~2.2]{FleschlerTolsaVilla}),
we find that, for any compact set $E\subset\mathbb R^n$,
\begin{align}\label{1003}
\operatorname{cap}\left(E\right)\sim\left[
\inf_{\mu}
\iint_{E\times E}
|X-Y|^{2-n}\,\mathrm d\mu\left(X\right)\,\mathrm d\mu\left(Y\right)
\right]^{-1},
\end{align}
where the infimum is taken over all probability measures $\mu$ supported
on $E$.
Let $\mu_R$ be the normalized Lebesgue measure on $\mathcal C_R$, i.e.,
$\,\mathrm d\mu_R(X):=2^{2-n}\pi^{-1} R^{2-n}\,\mathrm d X$
on $\mathcal C_R.$
Write $X=(z,y)$ and $Y=(z',y')$, and let
$a:=z-z'$ and $b:=y-y'$.  Since $|a|\le M_nR$, $|b|\le2$, and
$\mu_R$ is the normalized Lebesgue measure on $\mathcal C_R$, it follows that
\begin{align*}
\iint_{\mathcal C_R\times\mathcal C_R}
|X-Y|^{2-n}\,\mathrm d\mu_R\left(X\right)\,\mathrm d\mu_R\left(Y\right)
\le
M_nR^{-\left(n-2\right)}
\int_{|b|\le2}\int_{|a|\le M_nR}
\left(|a|^2+|b|^2\right)^{-\frac{n-2}{2}}\,\mathrm d a\,\mathrm d b.
\end{align*}
From polar coordinates in $\mathbb R^{n-2}$,
we deduce that, for $s\in(0,2]$,
\begin{align*}
\int_{|a|\le M_nR}
\left(|a|^2+s^2\right)^{-\frac{n-2}{2}}\,\mathrm d a
&\lesssim1+\log\frac{R}{s},
\end{align*}
which further implies that
\begin{align*}
\iint_{\mathcal C_R\times\mathcal C_R}
|X-Y|^{2-n}\,\mathrm d\mu_R\left(X\right)\,\mathrm d\mu_R\left(Y\right)
&\lesssim
R^{-\left(n-2\right)}
\int_0^2 s\left(1+\log\frac{R}{s}\right)\,\mathrm d s
\lesssim R^{-\left(n-2\right)}\log R.
\end{align*}
Combining this and \eqref{1003}, we complete the proof of
Lemma \ref{lem:codim-two-cylinder-capacity}.
\end{proof}

Now, we construct an example to prove
Theorem~\ref{thm:main}(ii), which shows the sharpness of $\alpha^{-2}$.

\begin{proof}[Proof of Theorem~\ref{thm:main}(ii)]
The upper bound for the universal Hardy constant follows from part~(i).
We construct domains proving the matching lower bound. Let $A_n\in(2\sqrt{n+2},\infty)$. By
Lemmas~\ref{1709} and \ref{lem:codim-two-cylinder-capacity}, we conclude that
there exists $b_n\in(0,\infty)$ such that, for every $R\in[2,\infty)$,
\begin{align}
\label{eq:cylinder-capacity-ratio}
\frac{\operatorname{cap}(\mathcal C_R)}{\operatorname{cap}(B(\mathbf 0,A_nR))}
\ge \frac{b_n}{\log R}.
\end{align}

Let $\gamma_n\in(0,\min\{\frac{b_n}{2},1\}).$
We first consider
$\alpha\in(0,\frac{\gamma_n}{4}]$. In this case, for simplicity, let
$L:=\frac{\gamma_n}{\alpha}$ and
$H:=e^{3L}$.
Then $L\in[4,\infty)$.
We define
\begin{align*}
\Omega_\alpha
:=\left(-2H,2H\right)^{n-2}\times
\left\{y\in\mathbb R^2:|y|\in\left(1,e^{2L}\right)\right\}
\end{align*}
and hence
$\mathbb R^{n-2}\times\overline{B_{\mathbb R^2}(\mathbf 0,1)}
\subset \Omega_\alpha^\complement=:F_\alpha$.
Clearly, $\Omega_\alpha$ is bounded, open, and connected.
Consider the compact set
\begin{align*}
K_\alpha
:=[-H,H]^{n-2}\times
\left\{y\in\mathbb R^2:|y|\in\left[e^{\frac{3L}{4}},e^L\right]\right\}
\Subset\Omega_\alpha.
\end{align*}
Fix $x=(z,y)\in K_\alpha$ and let
\begin{align*}
E_x
:=
\left\{\left(z',y'\right):\ \left|z'_i-z_i\right|\le|y|\ \text{for }i\in\{1,\ldots,n-2\},\
\left|y'\right|\le1\right\}.
\end{align*}
Then $E_x\subset F_\alpha$, and, since $|y|\in[2,\infty)$,
it follows that
\begin{align*}
\left|\left(z',y'\right)-\left(z,y\right)\right|^2
\le \left(n-2\right)|y|^2+\left(|y|+1\right)^2
\le\left(n+2\right)|y|^2,
\end{align*} which further implies that
$E_x\subset F_\alpha\cap B(x,A_n|y|).$
From the fact that the set $E_x$ is a translate of
$\mathcal C_{|y|}$,
$\log|y|\le L$, \eqref{eq:cylinder-capacity-ratio}, and
$L=\frac{\gamma_n}\alpha$, we infer that
\begin{align*}
\operatorname{cap}\left(\overline{F_\alpha\cap B\left(x,A_n|y|\right)}\right)
\ge\operatorname{cap}\left(E_x\right)\ge\frac{b_n}{L}\operatorname{cap}\left(B\left(\mathbf 0,A_n|y|\right)\right)
\ge\alpha\operatorname{cap}\left(B\left(\mathbf 0,A_n|y|\right)\right),
\end{align*}
and hence, for any $(z,y)\in K_\alpha$,
\begin{align}
\label{eq:dalpha-upper-sharp}
d_\alpha\left(z,y\right)\le A_n|y|.
\end{align}

Next, we construct the test function.  Choose
$\eta\in C_{\rm{c}}^\infty((-2,2)^{n-2})$ satisfying $0\le\eta\le1$ and
$\eta=1$ on $[-1,1]^{n-2}$, and let
$\eta_H:=\eta(\frac {\cdot}H).$
Define the test function
\begin{align*}
\phi_\alpha\left(t\right):=
\begin{cases}
0,&t\in\left(0,1\right],\\
\displaystyle\frac{\log t}{L},&t\in\left(1,e^L\right),\\
\displaystyle\frac{2L-\log t}{L},&t\in\left[e^L,e^{2L}\right),\\
0,&t\in\left[e^{2L},\infty\right).
\end{cases}
\end{align*}
Let $v_\alpha(z,y):=\eta_H(z)\phi_\alpha(|y|).$
Since $\eta_H$ is smooth and $y\mapsto\phi_\alpha(|y|)$ is Lipschitz
on $\{y\in\mathbb R^2:|y|\in(1,e^{2L})\}$, it follows that
$v_\alpha\in H^1(\Omega_\alpha)$. Moreover, using the fact
$\Omega_\alpha$ is a Lipschitz domain and $v_\alpha$ vanishes on its boundary
and the trace characterization of $H_0^1$ (see
\cite[Theorem 4.10]{NecasDirectMethods}), we conclude that
$v_\alpha\in H_0^1(\Omega_\alpha)$.
From Fubini's theorem and polar coordinates,
we infer that
\begin{align}\label{1047}
\int_{\Omega_\alpha}|\nabla v_\alpha\left(x\right)|^2\,\mathrm d x
&=\left(\int_{\mathbb R^{n-2}}|\nabla\eta_H|^2\,\mathrm d z
\right)\left[
\int_{\mathbb R^2}|\phi_\alpha\left(|y|\right)|^2\,\mathrm d y
\right]\nonumber\\
&\quad+\left(
\int_{\mathbb R^{n-2}}|\eta_H|^2\,\mathrm d z
\right)\left[
\int_{\mathbb R^2}|\nabla\phi_\alpha\left(|y|\right)|^2\,\mathrm d y
\right]\nonumber\\
&\lesssim\left(H^{n-4}e^{4L}+\frac{H^{n-2}}{L}
\right)=H^{n-2}\left(e^{-2L}+\frac1L\right)
\lesssim\frac{H^{n-2}}{L}.
\end{align}
For any $(z,y)\in K_\alpha$, by their definitions, we have $\eta_H(z)=1$ and
$\phi_\alpha(|y|)\ge\frac34$. Using \eqref{eq:dalpha-upper-sharp},
we find that
\begin{align*}
\int_{\Omega_\alpha}\frac{|v_\alpha(x)|^2}{d_\alpha(x)^2}\,\mathrm d x
\ge\int_{K_\alpha}\frac{|v_\alpha(x)|^2}{d_\alpha(x)^2}\,\mathrm d x
\gtrsim H^{n-2}
\int_{e^{\frac{3L}{4}}}^{e^L}\frac{r}{(A_nr)^2}\,\mathrm d r
\gtrsim H^{n-2}L,
\end{align*}
which, together with \eqref{1047}, further implies that
\begin{align}
\label{eq:sharpness-H01-ratio}
\int_{\Omega_\alpha}\frac{|v_\alpha(x)|^2}
{d_\alpha(x)^{2}}\,\mathrm d x
\left[\int_{\Omega_\alpha}|\nabla v_\alpha\left(x\right)|^2\,\mathrm d x\right]^{-1}
\gtrsim L^2
=\frac{\gamma_n^2}{\alpha^{2}}.
\end{align}
It remains only to replace $v_\alpha$ by a $C_{\rm{c}}^\infty(\Omega_\alpha)$ function.
Let $\{v_{\alpha,j}\}_{j\in\mathbb N}\subset
C_{\rm{c}}^\infty(\Omega_\alpha)$ converge to
$v_\alpha$ in $H_0^1(\Omega_\alpha)$.
Since $K_\alpha\Subset\Omega_\alpha$, it follows that
$\delta_\alpha:=\operatorname{dist}(K_\alpha,F_\alpha)\in(0,\infty).$
Moreover, for any $x\in K_\alpha$,
$d_\alpha(x)^{-2}\le \operatorname{dist}(x,F_\alpha)^{-2}\le\delta_\alpha^{-2}$.
Therefore,
\begin{align*}
\int_{K_\alpha}
\frac{|v_{\alpha,j}(x)|^2}{d_\alpha(x)^2}\,\mathrm d x
\to
\int_{K_\alpha}
\frac{|v_\alpha(x)|^2}{d_\alpha(x)^2}\,\mathrm d x
\end{align*}
and
\begin{align*}
\int_{\Omega_\alpha}|\nabla v_{\alpha,j}\left(x\right)|^2\,\mathrm d x
\to
\int_{\Omega_\alpha}|\nabla v_\alpha\left(x\right)|^2\,\mathrm d x
\end{align*}
as $j\to\infty$.
Thus, for some $j\in\mathbb N$, the ratio for $v_{\alpha,j}$ is
bounded below by a fixed positive dimensional multiple of $L^2$;
equivalently, \eqref{eq:sharpness-H01-ratio} holds with $v_\alpha$
replaced by $v_{\alpha,j}$ after possibly decreasing its implicit
dimensional constant.
Taking $u_\alpha:=v_{\alpha,j}$, we obtain the desired estimate for
$\alpha\in(0,\frac{\gamma_n}{4}]$.

It remains to consider
$\alpha\in(\frac{\gamma_n}{4},1]$. In this case,
let $\Omega_\alpha:=\Omega_0:=B(\mathbf0,1)$ and $F_0:=\Omega_0^\complement$, and fix a
nonzero function $u_0\in C_{\rm c}^\infty(\Omega_0)$.
For any $x\in\Omega_0$ and $r\in(1+|x|,\infty)$, choose
$s\in(1+|x|,r)$. Then
$\partial B(x,s)\subset\overline{F_0\cap B(x,r)}$, which, combined with
Lemma \ref{1709} and the
standard fact that a ball's and its sphere's Newtonian capacity
are the same, further implies that
\begin{align*}
\operatorname{cap}\left(\overline{F_0\cap B\left(x,r\right)}\right)
\ge \operatorname{cap}\left(\partial B\left(x,s\right)\right)=\operatorname{cap}\left(B\left(x,s\right)\right)
=\kappa_ns^{n-2}.
\end{align*}
Letting $s\to r$ and using Lemma \ref{1709} again, we conclude that
\begin{align*}
\operatorname{cap}\left(\overline{F_0\cap B\left(x,r\right)}\right)
\ge\kappa_nr^{n-2}
=\operatorname{cap}\left(B\left(\mathbf0,r\right)\right).
\end{align*}
Thus, for every $\alpha\in(0,1]$,
$d_{\alpha,\Omega_0}(x)\le1+|x|$ and $1+|x|\in(0,2)$.
Let
\begin{align*}
c_{0,n}:=
\int_{\Omega_0}|u_0\left(x\right)|^2\,\mathrm d x
\left[4\displaystyle\int_{\Omega_0}|\nabla u_0\left(x\right)|^2\,\mathrm d x\right]^{-1}
\in\left(0,\infty\right).
\end{align*}
Then
\begin{align*}
\int_{\Omega_0}
\frac{|u_0(x)|^2}{d_{\alpha,\Omega_0}(x)^2}\,\mathrm d x
\ge c_{0,n}
\int_{\Omega_0}|\nabla u_0\left(x\right)|^2\,\mathrm d x.
\end{align*}
Since $\alpha^{-2}\le16\gamma_n^{-2}$, it follows that
\begin{align*}
\int_{\Omega_0}
\frac{|u_0(x)|^2}{d_{\alpha,\Omega_0}(x)^2}\,\mathrm d x
\ge
\frac{c_{0,n}\gamma_n^2}{16}\alpha^{-2}
\int_{\Omega_0}|\nabla u_0\left(x\right)|^2\,\mathrm d x.
\end{align*}
By homogeneity, normalizing the above two test functions gives the desired estimate.
This completes the proof of Theorem~\ref{thm:main}\textup{(ii)}.
\end{proof}

\begin{remark}
Maz'ya and Shubin \cite{MazyaShubinDrum} introduced the interior
capacitary radius
\begin{align}\label{2043}
r_{\Omega,\alpha}:=\sup\left\{r\in\left(0,\infty\right):
\begin{array}{l}
\text{there exists a ball }B\left(x,r\right)\text{ such that}\\
\operatorname{cap}\left(\overline{B\left(x,r\right)}\setminus\Omega\right)
\le\alpha\operatorname{cap}\left(\overline{B\left(0,r\right)}\right)
\end{array}\right\},
\end{align}
Their spectral estimate gives a Hardy-type bound with the constant
weight $r_{\Omega,\alpha}^{-2}$ and a coefficient of order
$\alpha^{-1}$. The present theorem retains the full spatial
variation of $d_\alpha^{-2}$; its universal coefficient has the
necessarily larger order $\alpha^{-2}$. These are different
statements. In particular, the sharpness of the local weighted
law does not assert sharpness of an $\alpha^{-2}$ global
spectral-gap bound; see Corollary~\ref{cor:uniform-survival}.
\end{remark}

\section{Diffusive survival, quantum stability, and Fisher information}
\label{sec:physical-consequences}

We now derive consequences of the trapping principle. They require
no smoothness of $\partial\Omega$. The diffusion and quantum models
share the same Dirichlet form: the heat evolution describes killed Brownian motion, while the
self-adjoint Hamiltonian generates unitary quantum evolution with
Dirichlet walls.
The absorption interpretation refers to the diffusion and to the
associated imaginary-time semigroup.

This section is organized as follows. Subsection~\ref{6.1} extends the
basic estimates to finite-energy states. Subsection~\ref{6.2} gives an
explicit spherical example illustrating the absorption scale.
Subsection~\ref{6.3} derives survival decay and a corresponding
spectral lower bound under uniform trapping, including a recovery of
the Maz'ya--Shubin lower bound, while
Subsection~\ref{6.4} gives a weighted heat-dissipation estimate without
such a uniform bound. Subsection~\ref{6.5} interprets the Hardy estimate
as a stability bound for an attractive quantum potential.
Subsection~\ref{6.6} reformulates the same estimate in terms of Fisher
information and concentration near traps. Finally, Subsection~\ref{6.7}
records the corresponding realization-wise bounds for random
absorbing media.

\subsection{Extension to finite-energy states}\label{6.1}

We first extend the Hardy-type estimates from smooth test
functions to the full finite-energy space $H_0^1(\Omega)$.

\begin{proposition}\label{prop:form-extension}
Let $n\in[3,\infty)\cap\mathbb N$, let
$\Omega\subset\mathbb R^n$ be open, and let
$\alpha\in(0,1]$ and $p\in(0,1)$. Let $d_\alpha$ be the capacitary
distance in \eqref{eq:dalpha-def}, write $V_\alpha:=d_\alpha^{-2}$,
and let $C_n$ be the constant in \eqref{eq:intro-constants}.
Then it holds, for every real or complex $u\in H_0^1(\Omega)$,
\begin{align}
\int_\Omega V_\alpha\left(x\right)|u\left(x\right)|^2\,\mathrm d x
\le\frac{C_n}{\alpha^2}\mathcal E_\Omega\left(u,u\right)
\label{eq:form-hardy}
\end{align}
and
\begin{align}
\int_\Omega\frac{|u\left(x\right)|^2}{\tau_p\left(x\right)}\,\mathrm d x
\le\frac{2}{p^2}\mathcal E_\Omega\left(u,u\right).
\label{eq:form-quantile}
\end{align}
In particular, multiplication by $V_\alpha^{\frac{1}{2}}$ and
$\tau_p^{-\frac{1}{2}}$ is continuous from $H_0^1(\Omega)$ to $L^2(\Omega)$.
\end{proposition}

\begin{proof}
First assume that $u$ is real and choose
$u_j\in C_{\mathrm{c}}^\infty(\Omega)$ converging to $u$ in $H^1$.
For either weight, the corresponding Hardy inequality applied to
$u_j-u_k$ shows that the weighted sequence is Cauchy in $L^2$.
A subsequence converges to $u$ almost everywhere, identifying the
weighted limit. Passing to the limit and applying Theorems \ref{thm:main}(i)
and \ref{thm:exit-distance-hardy} then give the claimed bounds.
For complex states, apply the real-valued inequalities, respectively, to the real
and imaginary parts and add them. The energy of a complex state is
$\int_\Omega|\nabla u\left(x\right)|^2\,\mathrm d x$. This completes the proof.
\end{proof}

\subsection{A solvable spherical absorber}\label{6.2}

Consider $\Omega:=\mathbb R^3\setminus\overline{B(0,a)}$, with $a\in(0,\infty)$,
and write $R:=|x|\in(a,\infty)$. Its absorber has capacity $4\pi a$.
For the normalization \eqref{eq:intro-brownian} and any
$t\in(0,\infty)$, the cumulative absorption probability is
\begin{align}\label{eq:spherical-trap-cdf}
\mathbb P_x\left(\tau_\Omega\le t\right)
=\frac{a}{R}\operatorname{erfc}\left(\frac{R-a}{2\sqrt t}\right),
\end{align}
where
\begin{align*}
\operatorname{erfc}\left(s\right):=\frac{2}{\sqrt\pi}
\int_s^\infty e^{-v^2}\,\mathrm d v.
\end{align*}
To verify the formula \eqref{eq:spherical-trap-cdf}, let $q(t,R)$ denote the right side of
\eqref{eq:spherical-trap-cdf}. The radial backward heat equation
is $q_t=q_{RR}+\frac{2q_R}{R}$. The function $v(t,R):=Rq(t,R)$ solves
$v_t=v_{RR}$ for $R\in(a,\infty)$, with $v(t,a)=a$ and $v(0,R)=0$.
The displayed complementary-error-function solution has these
conditions, and $q$ is the bounded solution with absorbing
boundary value one; the probabilistic representation identifies
it with the absorption probability, which gives \eqref{eq:spherical-trap-cdf}.

From \eqref{eq:spherical-trap-cdf}, it follows that, for any $x\in\Omega$,
\begin{align*}
\tau_p\left(x\right)=
\begin{cases}
\displaystyle\frac{(R-a)^2}
{4[\operatorname{erfc}^{-1}(\frac{pR}{a})]^2},
&\displaystyle p\in\left(0,\frac{a}{R}\right),\\[2mm]
\infty,&\displaystyle p\in\left[\frac{a}{R},1\right).
\end{cases}
\end{align*}
The eventual absorption probability is $\frac{a}{R}$, and a particle
escapes forever with probability $1-\frac{a}{R}$. Thus, its mean exit
time, as an extended expectation, is infinite, while every
quantile below $\frac{a}{R}$ is finite. This explains why we
use $\tau_p$, rather than its expectation, in Theorems
\ref{thm:exit-distance-hardy} and \ref{thm:compact-killing}.

At radius $r:=R+a$, the compact closure
$\overline{F\cap B(x,r)}$ contains the entire absorber. The
capacity ratio in \eqref{eq:intro-trapping-scale} is $\frac{a}{R+a}$.
Therefore, for any $\alpha\in(0,\frac{a}{R+a}]$,
$d_\alpha(x)\le R+a.$
This exactly solvable model illustrates how the local
electrostatic input selects a finite absorption quantile $d_\alpha(x)$, even
when a mean exit time cannot supply a finite time scale.
Thus, $d_\alpha$ is the natural spatial scale in
Theorem~\ref{thm:main}.

\subsection{Uniform capacitary loading and survival decay}\label{6.3}

For $D\in(0,\infty)$, let $X_t^{(D)}:=x+\sqrt{2D}\,W_t$ and denote its exit
time by $\tau_\Omega^{(D)}$. Its killed semigroup is
$P_{Dt}^\Omega=e^{-DtA_\Omega}$; hence its absorption-quantile time
is $\frac{\tau_p(x)}{D}$. This is the physical time conversion behind the
following result.

\begin{corollary}
\label{cor:uniform-survival}
Let $n\in[3,\infty)\cap\mathbb N$,
$\Omega\subset\mathbb R^n$ be open, $R\in(0,\infty)$,
and $\alpha\in(0,1]$. Let
$a_n$ and $T_n$ be the dimensional constants in
\eqref{eq:intro-constants} and $A_\Omega:=-\Delta_\Omega^D$.
Assume $d_\alpha(x)\le R$ for every $x\in\Omega$.
Let $p_\alpha:=\frac{a_n\alpha}{2}$ and $t_0:=\frac{T_nR^2}{D}.$
Then, for every $t\in[0,\infty)$ and $x\in\Omega$,
\begin{align}\label{eq:uniform-survival}
\mathbb P_x\left(\tau_\Omega^{\left(D\right)}>t\right)
\le\left(1-p_\alpha\right)^{\lfloor \frac{t}{t_0}\rfloor}.
\end{align}
Writing $\lambda_0(\Omega):=\inf\sigma(A_\Omega)$, where
$\sigma(A_\Omega)$ denotes the spectrum of $A_\Omega$, we also have
\begin{align}\label{eq:uniform-gap}
\lambda_0\left(\Omega\right)
\ge\frac{-\log(1-p_\alpha)}{T_nR^2}
\ge\frac{a_n\alpha}{2T_nR^2}.
\end{align}
If $d_\alpha\le R$ only almost everywhere, the survival statement
holds almost everywhere and the spectral statement is unchanged.
\end{corollary}

\begin{proof}
Corollary~\ref{cor:scale-comparison} gives
$P_{T_n d_\alpha(x)^2}^\Omega\mathbf1(x)\le1-p_\alpha$.
Monotonicity in time therefore gives
$P_{Dt_0}^\Omega\mathbf1\le1-p_\alpha$. Positivity and the semigroup
property imply, inductively, that, for any
$k\in[0,\infty)\cap\mathbb Z$,
\begin{align*}
P_{kDt_0}^\Omega\mathbf1\le\left(1-p_\alpha\right)^k.
\end{align*}
This proves \eqref{eq:uniform-survival}, again by time monotonicity.
The same argument works almost everywhere because the killed
transition kernels are absolutely continuous.

The one-step bound gives
$\|P_{T_nR^2}^\Omega\|_{L^\infty(\Omega)\to L^\infty(\Omega)}
\le1-p_\alpha$.
Symmetry gives the same $L^1$ bound, and interpolation gives
\begin{align*}
\left\|e^{-T_nR^2A_\Omega}\right\|_{L^2\left(\Omega\right)\to L^2\left(\Omega\right)}\le1-p_\alpha.
\end{align*}
By the spectral theorem the left side is
$e^{-T_nR^2\lambda_0(\Omega)}$. Taking logarithms proves the first
inequality in \eqref{eq:uniform-gap}; the second follows from
$-\log(1-p_\alpha)\ge p_\alpha$. This completes the proof.
\end{proof}

Thus, a medium with a uniformly bounded trapping scale has repeated
opportunities to absorb a particle. The global spectral estimate
has linear dependence on $\alpha$. This is consistent with the
global capacitary-radius estimates of \cite{MazyaShubinDrum}; the
quadratic loss in Theorem~\ref{thm:main} concerns the full local
weight, rather than a single uniform confinement scale.

In addition, the Hardy inequality in Theorem~\ref{thm:main}\textup{(i)}
immediately gives
$\lambda_0(\Omega)\ge\frac{\alpha^2}{C_n}\,r_{\Omega,\alpha}^{-2},$
where $r_{\Omega,\alpha}$ is as in \eqref{2043}.
Hence, for each fixed $\alpha$, the Maz'ya--Shubin lower bound is
a direct consequence of the capacitary Hardy inequality, in
the same spirit in which Frank and Larson \cite{FrankLarson2021}
deduced Lieb's estimate from Davies' Hardy inequality. The following
corollary improves the dependence on $\alpha$ from quadratic to
linear, in agreement with the order in \cite{MazyaShubinDrum}.

\begin{corollary}\label{cor:mazya-shubin-lower}
Let $n\in[3,\infty)\cap\mathbb N$,
$\Omega\subset\mathbb R^n$ be a nonempty open set, and
$\alpha\in(0,1)$. Let $r_{\Omega,\alpha}$ be the same as in \eqref{2043} and
$a_n$ and $T_n$ be the dimensional constants in
\eqref{eq:intro-constants}.
If $r_{\Omega,\alpha}<\infty$, then
\begin{align}\label{eq:mazya-shubin-from-trapping}
\lambda_0(\Omega)
\ge \frac{a_n\alpha}{2T_n}\,r_{\Omega,\alpha}^{-2}.
\end{align}
In particular, $\lambda_0(\Omega)>0$ whenever
$r_{\Omega,\alpha}<\infty$.
\end{corollary}

\begin{proof}
Fix $R\in(r_{\Omega,\alpha},\infty)$ and $x\in\Omega$. For every
$s\in(r_{\Omega,\alpha},R)$, the definition \eqref{2043} of
$r_{\Omega,\alpha}$ gives
\begin{align*}
\operatorname{cap}\left(\overline{B(x,s)}\setminus\Omega\right)
>\alpha\operatorname{cap}\left(\overline{B(\mathbf0,s)}\right).
\end{align*}
From this,
$\overline{B(x,s)}\setminus\Omega\subset
F\cap B(x,R)\subset\overline{F\cap B(x,R)}$, Lemma~\ref{1709},
and the monotonicity of capacity, we infer that
\begin{align*}
\operatorname{cap}\left(\overline{F\cap B(x,R)}\right)
\ge \alpha\kappa_n s^{n-2}.
\end{align*}
Letting $s\to R^-$ yields $d_\alpha(x)\le R$.
Using Corollary~\ref{cor:uniform-survival} and
letting $R\to r_{\Omega,\alpha}^+$, we conclude that
\eqref{eq:mazya-shubin-from-trapping} holds.
This completes the proof of
Corollary \ref{cor:mazya-shubin-lower}.
\end{proof}

\subsection{Weighted heat dissipation}\label{6.4}

Even when the trapping scale is unbounded, the local law controls
how a diffusing amplitude dissipates energy. Let
$u(t):=e^{-DtA_\Omega}u_0$ for $u_0\in L^2(\Omega)$.

\begin{corollary}
Let $n\in[3,\infty)\cap\mathbb N$,
$\Omega\subset\mathbb R^n$ be open,
$\alpha\in(0,1]$, $p\in(0,1)$, and $D\in(0,\infty)$. Let
$A_\Omega:=-\Delta_\Omega^D$, $u_0\in L^2(\Omega)$, and
$u(t):=e^{-DtA_\Omega}u_0$. Let $V_\alpha$ and $\tau_p$ be as in
Proposition~\ref{prop:form-extension}, and let $C_n$ be the constant
in \eqref{eq:intro-constants}. Then, for every $S\in(0,\infty)$,
\begin{align}\label{e6.6}
\int_0^S\int_\Omega V_\alpha\left(x\right)|u\left(t,x\right)|^2\,\mathrm d x\,\mathrm d t
\le\frac{C_n}{2D\alpha^2}
\left[\|u_0\|_{L^2\left(\Omega\right)}^2
-\|u\left(S\right)\|_{L^2\left(\Omega\right)}^2\right]
\end{align}
and
\begin{align}\label{e6.7}
\int_0^S\int_\Omega\frac{|u(t,x)|^2}{\tau_p(x)}\,\mathrm d x\,\mathrm d t
\le\frac{1}{Dp^2}
\left[\|u_0\|_{L^2\left(\Omega\right)}^2
-\|u\left(S\right)\|_{L^2\left(\Omega\right)}^2\right].
\end{align}
In particular, the left sides over $t\in(0,\infty)$ are bounded by
$\frac{C_n\|u_0\|_{L^2(\Omega)}^2}{2D\alpha^2}$ and
$\frac{\|u_0\|_{L^2(\Omega)}^2}{Dp^2}$,
respectively.
\end{corollary}

\begin{proof}
The spectral theorem gives $u(t)\in H_0^1(\Omega)$ for $t\in(0,\infty)$ and
the energy identity
\begin{align*}
2D\int_0^S\mathcal E_\Omega\left(u\left(t\right),u\left(t\right)\right)\,\mathrm d t
=\|u_0\|_{L^2\left(\Omega\right)}^2
-\|u\left(S\right)\|_{L^2\left(\Omega\right)}^2.
\end{align*}
Integrating \eqref{eq:form-hardy} and \eqref{eq:form-quantile} in
time and using this identity yield, respectively,
\eqref{e6.6} and \eqref{e6.7}. Monotone convergence
further gives their infinite-time estimates.
\end{proof}

The spatial weight remains present in these bounds. They quantify
the integrated squared amplitude in regions with a short trapping
or absorption-quantile scale, without replacing those scales by a
global maximum.

\subsection{Quantum kinetic cost and attractive potentials}\label{6.5}

Let $\hbar,m\in(0,\infty)$. For a normalized quantum state
$\psi\in H_0^1(\Omega;\mathbb C)$, define
\begin{align*}
\mathcal T[\psi]:=\frac{\hbar^2}{2m}\mathcal E_\Omega\left(\psi,\psi\right)
\quad\mathrm{and}\quad
\mathcal V_\alpha\left(x\right):=\frac{\hbar^2}{2m}V_\alpha\left(x\right).
\end{align*}
Proposition~\ref{prop:form-extension} gives the kinetic-cost bound
\begin{align*}
\int_\Omega\mathcal V_\alpha\left(x\right)|\psi\left(x\right)|^2\,\mathrm d x
\le\frac{C_n}{\alpha^2}\mathcal T[\psi].
\end{align*}
Both sides have energy units. The field $\mathcal V_\alpha$ is an
energy scale constructed from the geometry. When it is inserted
as an attractive potential, the following statement controls its
competition with kinetic energy.

\begin{corollary}
Let $n\in[3,\infty)\cap\mathbb N$,
$\Omega\subset\mathbb R^n$ be open,
$\alpha\in(0,1]$, $\hbar,m\in(0,\infty)$, and $g\in\mathbb R$.
Let $V_\alpha$ be as in Proposition~\ref{prop:form-extension} and
$C_n$ be the constant in \eqref{eq:intro-constants}.
On $H_0^1(\Omega;\mathbb C)$ define
\begin{align*}
\mathfrak q_{\alpha,g}[\psi]
:=\frac{\hbar^2}{2m}
\left[\mathcal E_\Omega\left(\psi,\psi\right)
-g\int_\Omega V_\alpha\left(x\right)|\psi\left(x\right)|^2\,\mathrm d x\right].
\end{align*}
Then it holds, for $g\in[0,\frac{\alpha^2}{C_n}]$,
\begin{align}\label{eq:schrodinger-coercivity}
\mathfrak q_{\alpha,g}[\psi]
\ge\left(1-\frac{gC_n}{\alpha^2}\right)\mathcal T[\psi]
\ge0
\end{align}
and, if $g\in[0,\frac{\alpha^2}{C_n})$, this densely defined form is closed and defines
a nonnegative self-adjoint operator, denoted by
\begin{align*}
H_{\alpha,g}
=-\frac{\hbar^2}{2m}\Delta_\Omega^D
-g\frac{\hbar^2}{2m}V_\alpha,
\end{align*}
where the expression is understood in the sense of forms.
At $g=\frac{\alpha^2}{C_n}$, only form nonnegativity is asserted here.
\end{corollary}

\begin{proof}
Inequality \eqref{eq:schrodinger-coercivity} is an immediate
consequence of \eqref{eq:form-hardy}. For strict inequality,
the form norm
$\{\mathfrak q_{\alpha,g}[\psi]+\|\psi\|_{L^2(\Omega)}^2\}^{\frac{1}{2}}$
is equivalent to the $H_0^1$ norm of $\psi$, because the potential is
nonnegative and its relative form bound is smaller than one.
Completeness of $H_0^1$ proves closedness. The representation
theorem for closed nonnegative forms then gives the operator.
\end{proof}

The coupling scale in this corollary inherits the sharpness of the
Hardy law. Define
\begin{align*}
g_{\mathrm c}\left(\Omega,\alpha\right)
:=\inf_{\genfrac{}{}{0pt}{}{u\in H_0^1\left(\Omega\right),\ u\ne0}{
                   \int_\Omega V_\alpha\left(x\right)|u\left(x\right)|^2\,\mathrm d x\in\left(0,\infty\right)}}
\frac{\mathcal E_\Omega(u,u)}{\int_\Omega V_\alpha\left(x\right)|u\left(x\right)|^2\,\mathrm d x},
\end{align*}
with the infimum $\infty$ if there is no such $u$. Allowing complex
functions does not change the infimum, as follows by decomposing a
function into its real and imaginary parts. Thus, $g_{\mathrm c}$ is precisely
the largest coupling for which the form is nonnegative on its
stated domain. Theorem~\ref{thm:main} gives
\begin{align*}
\frac{\alpha^2}{C_n}
\le\inf_{\Omega\ \mathrm{open}}g_{\mathrm c}\left(\Omega,\alpha\right)
\le\frac{\alpha^2}{c_n}.
\end{align*}
Thus, the geometry-independent coupling threshold has order
$\alpha^2$. This concerns nonnegativity of a one-particle form;
it does not identify the exact threshold for an individual medium.

\subsection{Fisher information and concentration near traps}\label{6.6}

Let $f$ be a nonnegative probability density on $\Omega$ such that
$\sqrt f\in H_0^1(\Omega)$. Its \emph{Fisher information} is
\begin{align*}
I\left(f\right):=4\int_\Omega|\nabla\sqrt f\left(x\right)|^2\,\mathrm d x.
\end{align*}
For a differentiable density this agrees with
$\int_{\{x\in\Omega:f(x)\in(0,\infty)\}}\frac{|\nabla f\left(x\right)|^2}{f\left(x\right)}\,\mathrm d x$, with the usual weak
interpretation at zeros. The condition on $\sqrt f$ specifies the
finite-energy boundary behavior required by the Dirichlet model.

\begin{corollary}
\label{cor:fisher-trapping}
Let $n\in[3,\infty)\cap\mathbb N$,
$\Omega\subset\mathbb R^n$ be open,
$\alpha\in(0,1]$, and $p\in(0,1)$. Let $f$ be a nonnegative
probability density on $\Omega$ such that
$\sqrt f\in H_0^1(\Omega)$. Let $X$ be an
$\Omega$-valued random variable with density $f$ such that
$\mathbb E_f[h(X)]:=\int_\Omega h(x)f(x)\,\mathrm d x$. Let
$d_\alpha$ and $\tau_p$ be as in
Proposition~\ref{prop:form-extension}, and let $C_n$ be the constant
in \eqref{eq:intro-constants}.
Then
\begin{align}
\mathbb E_f[d_\alpha\left(X\right)^{-2}]
\le\frac{C_n}{4\alpha^2}I\left(f\right)
\label{eq:fisher-capacitary}
\end{align}
and
\begin{align}\label{e6.10}
\mathbb E_f[\tau_p\left(X\right)^{-1}]
\le\frac{1}{2p^2}I\left(f\right).
\end{align}
Consequently, for $r,t\in(0,\infty)$,
\begin{align}\label{e6.11}
\int_{\{d_\alpha\le r\}} f\,\mathrm dx
\le\frac{C_nr^2}{4\alpha^2}I\left(f\right)
\quad and\quad
\int_{\{\tau_p\le t\}} f\,\mathrm dx
\le\frac{t}{2p^2}I\left(f\right).
\end{align}
\end{corollary}

\begin{proof}
Applying Proposition~\ref{prop:form-extension} to $u=\sqrt f$
respectively yields \eqref{eq:fisher-capacitary} and \eqref{e6.10}.
For the concentration bounds \eqref{e6.11}, use
$\mathbf1_{\{d_\alpha\le r\}}\le r^2d_\alpha^{-2}$ and
$\mathbf1_{\{\tau_p\le t\}}\le t\tau_p^{-1}$ and integrate
against $f$.
\end{proof}

For example, if $f$ assigns probability at least $\eta\in(0,\infty)$ to
$\{d_\alpha\le r\}$, then
$I\left(f\right)\ge\frac{4\alpha^2\eta}{C_nr^2}.$
The capacitary Fisher inequality has optimal order
$\alpha^{-2}$ over all media and admissible densities: take
$f:=\frac{|u_\alpha|^2}{\|u_\alpha\|_{L^2(\Omega_\alpha)}^2}$
with $u_\alpha$ as in the sharpness construction of Theorem \ref{thm:main}(ii).
Since $|u_\alpha|\in H_0^1$ and
$\int_{\Omega}|\nabla|u_\alpha||^2=\int_{\Omega}|\nabla u_\alpha|^2$ for real
$u_\alpha$, the ratio in \eqref{eq:fisher-capacitary} is one
quarter of the Hardy ratio.

Fisher information is a diffusion-associated information functional,
so these bounds give a statistical reading of the energy principle.
They do not require an entropy balance for a surviving density.
Such a balance in an absorbing medium also involves loss of mass
and its boundary behavior.

\subsection{Realization-wise bounds in a random absorbing medium}\label{6.7}

Suppose that, for every sample point $\omega$,
$\Omega(\omega)$ is a random open set and, for almost every
sample point $\omega$, a density $f_\omega$ satisfies all the hypotheses of
Corollary~\ref{cor:fisher-trapping} on $\Omega(\omega)$. If the
integrands and $I(f_\omega)$ are measurable, that corollary and
Tonelli's theorem give
\begin{align}\label{eq:averaged-information}
\mathbb E_\omega\!\left[
\int_{\Omega\left(\omega\right)}
\frac{f_\omega(x)}{d_{\alpha,\Omega(\omega)}(x)^2}\,\mathrm d x\right]
\le\frac{C_n}{4\alpha^2}\mathbb E_\omega[I\left(f_\omega\right)].
\end{align}
Likewise, a common bound $d_{\alpha,\Omega(\omega)}\le R$ permits
the same deterministic survival estimate
\eqref{eq:uniform-survival} in each realization. These are direct
consequences of uniformity over geometries. The disorder average
in \eqref{eq:averaged-information} uses the same deterministic
constant. Quantitative asymptotics for a particular random medium
still require additional information about its distribution.

\medskip

\noindent\textbf{Acknowledgements}\quad
The authors acknowledge the use of AI tools during the exploratory stage of this project.
All mathematical arguments and proofs in the final manuscript were checked
and written by the authors.

The authors would also like to thank Professor Rupert L. Frank for pointing out
the connection with the Maz'ya--Shubin result \cite{MazyaShubinDrum} 
and providing inspiring observation leading to Corollary~\ref{cor:mazya-shubin-lower}.

\bigskip

\noindent Yiqun Chen, Dachun Yang, Wen Yuan, and Yangyang Zhang
%(Corresponding author)

\smallskip

\noindent  Laboratory of Mathematics and Complex Systems
(Ministry of Education of China),
School of Mathematical Sciences, Institute for Advanced Study,
Beijing Normal University,
Beijing 100875, The People's Republic of China

\smallskip

\noindent {\it E-mails}: \texttt{yiqunchen@mail.bnu.edu.cn} (Y. Chen)

\noindent\phantom{{\it E-mails:}} \texttt{dcyang@bnu.edu.cn} (D. Yang)

\noindent\phantom{{\it E-mails:}} \texttt{wenyuan@bnu.edu.cn} (W. Yuan)

\noindent\phantom{{\it E-mails:}} \texttt{yangyzhang@bnu.edu.cn} (Y. Zhang)

\bigskip

\noindent Jie Xiao

\smallskip

\noindent Department of Mathematics and Statistics, Memorial 
University, St. John's, NL A1C5S7, Canada

\smallskip

\noindent {\it E-mail}: \texttt{jxiao@mun.ca}


\begin{thebibliography}{99}

\bibitem{AghajaniKinnunenRadulescu2025}
A. Aghajani, J. Kinnunen and V. D. R{\u{a}}dulescu,
Hardy-type inequalities for the drifting $p$-Laplace operator
and applications,
Proc. R. Soc. Edinb. Sect. A Math. (2025), https://doi.org/10.1017/prm.2025.14.

\vspace{-0.3cm}

\bibitem{AnconaHardy}
A. Ancona,
On strong barriers and an inequality of Hardy for domains in $\mathbb R^n$,
J. London Math. Soc. (2) 34 (1986), 274--290.

\vspace{-0.3cm}

\bibitem{ArendtGoldsteinGoldstein2006}
W. Arendt, G. R. Goldstein and J. A. Goldstein,
Outgrowths of Hardy's inequality,
in: Recent Advances in Differential Equations and Mathematical Physics, pp. 51--68,
Contemp. Math. 412,
American Mathematical Society, Providence, RI, 2006.

\vspace{-0.3cm}

\bibitem{BarasGoldstein1984}
P. Baras and J. A. Goldstein,
The heat equation with a singular potential,
Trans. Amer. Math. Soc. 284 (1984), 121--139.

\vspace{-0.3cm}

\bibitem{BetsakosBoudabraMarkowsky}
D. Betsakos, M. Boudabra and G. Markowsky,
On the duration of stays of Brownian motion in domains in Euclidean space,
Electron. Commun. Probab. 27 (2022), Paper No. 58, 12 pp.

\vspace{-0.3cm}

\bibitem{Chua2005}
S.-K. Chua,
Sharp conditions for weighted Sobolev interpolation inequalities,
Forum Math. 17 (2005), 461--478.

\vspace{-0.3cm}

\bibitem{ChungZhao}
K. L. Chung and Z. X. Zhao,
From Brownian Motion to Schr\"odinger's Equation,
Grund-lehren der mathematischen Wissenschaften, vol. 312,
Springer-Verlag, Berlin, 1995.

\vspace{-0.3cm}

\bibitem{CohnMeasureTheory}
D. L. Cohn,
Measure Theory,
second edition, Birkh\"auser Advanced Texts,
Birkh\"auser/ Springer, New York, 2013.

\vspace{-0.3cm}

\bibitem{DominguezTikhonov2019}
O. Dom\'inguez and S. Tikhonov,
Sobolev embeddings, extrapolations, and related inequalities,
Memoirs of the European Mathematical Society (to appear) or
arXiv:1909.12818.

\vspace{-0.3cm}

\bibitem{df12} B. Dyda and R. L. Frank,
Fractional Hardy--Sobolev--Maz'ya inequality for domains,
Studia Math. 208 (2012), 151–166.

\vspace{-0.3cm}

\bibitem{EvansPDE}
L.C. Evans,
Partial Differential Equations,
second edition, Graduate Studies in Mathematics,
vol. 19, American Mathematical Society, Providence, RI, 2010.

\vspace{-0.3cm}

\bibitem{EvansGariepy}
L. C. Evans and R. F. Gariepy,
Measure Theory and Fine Properties of Functions,
revised edition, Textbooks in Mathematics,
CRC Press, Boca Raton, FL, 2015.

\vspace{-0.3cm}

\bibitem{Fitzsimmons2000}
P. J. Fitzsimmons,
Hardy's inequality for Dirichlet forms,
J. Math. Anal. Appl. 250 (2000), 548--560.

\vspace{-0.3cm}

\bibitem{FleschlerTolsaVilla}
I. Fleschler, X. Tolsa and M. Villa,
Carleson's $\varepsilon^2$ conjecture in higher dimensions,
Invent. Math. 241 (2025), 207--307.

\vspace{-0.3cm}

\bibitem{FrankLarson2021}
R. L. Frank and S. Larson,
Two consequences of Davies' Hardy inequality,
Funct. Anal. Appl. 55 (2021), 174--177.

\vspace{-0.3cm}

\bibitem{fs08} R. L. Frank and R. Seiringer,
Non-linear ground state representations and sharp Hardy inequalities,
Funct. Anal. Appl. 55 (2021), 174--177
(Translation of Funktsional. Anal. i Prilozhen.
55 (2021), 118--121).


\vspace{-0.3cm}

\bibitem{FukushimaOshimaTakeda}
M. Fukushima, Y. \=Oshima and M. Takeda,
Dirichlet Forms and Symmetric Markov Processes,
De Gruyter Studies in Mathematics 19,
Walter de Gruyter, Berlin, 1994.

\vspace{-0.3cm}

\bibitem{GunawanHakimNakaiSawano2018}
H. Gunawan, D. I. Hakim, E. Nakai and Y. Sawano,
The Hardy and Heisenberg inequalities in Morrey spaces,
Bull. Aust. Math. Soc. 97 (2018), 480--491.

\vspace{-0.3cm}

\bibitem{Hardy1920}
G. H. Hardy,
Note on a theorem of Hilbert,
Math. Z. 6 (1920), 314--317.

\vspace{-0.3cm}

\bibitem{Hardy1925}
G. H. Hardy,
An inequality between integrals,
Messenger Math. 54 (1925), 150--156.

\vspace{-0.3cm}

\bibitem{KalamajskaPietruskaPaluba2011}
A. Ka{\l}amajska and K. Pietruska-Pa{\l}uba,
On a variant of the Gagliardo--Nirenberg inequality deduced
from the Hardy inequality,
Bull. Pol. Acad. Sci. Math. 59 (2011), 133--149.

\vspace{-0.3cm}

\bibitem{KaratzasShreve}
I. Karatzas and S. E. Shreve,
Brownian Motion and Stochastic Calculus,
second edition, Graduate Texts in Mathematics 113,
Springer, New York, 1991.

\vspace{-0.3cm}

\bibitem{KinnunenMartio1997}
J. Kinnunen and O. Martio,
Hardy's inequalities for Sobolev functions,
Math. Res. Lett. 4 (1997), 489--500.

\vspace{-0.3cm}

\bibitem{KinnunenKorteHardy}
J. Kinnunen and R. Korte,
Characterizations for the Hardy inequality,
in: Around the Research of Vladimir Maz'ya. I, pp.
239--254,
Int. Math. Ser. (N. Y.), 11,
Springer, New York, 2010.

\vspace{-0.3cm}

\bibitem{Lieb1983}
E. H. Lieb,
On the lowest eigenvalue of the Laplacian for the intersection of two domains,
Invent. Math. 74 (1983), 441--448.

\vspace{-0.3cm}

\bibitem{LewisUniformlyFat}
J. L. Lewis,
Uniformly fat sets,
Trans. Amer. Math. Soc. 308 (1988), 177--196.

\vspace{-0.3cm}

\bibitem{LuSteinerbergerDV}
J. Lu and S. Steinerberger,
A variation on the Donsker--Varadhan inequality for the principal eigenvalue,
Proc. A 473 (2017), 20160877, 6 pp.

\vspace{-0.3cm}

\bibitem{MazyaIsocapacitary}
V. Maz'ya,
Lectures on isoperimetric and isocapacitary inequalities
in the theory of Sobolev spaces,
in: Heat Kernels and Analysis on Manifolds, Graphs, and Metric Spaces,
307--340,
Contemp. Math., 338, Amer. Math. Soc., Providence, RI, 2003.

\vspace{-0.3cm}

\bibitem{MazyaProblems}
V. Maz'ya,
Seventy five (thousand) unsolved problems in analysis and partial differential equations,
Integral Equations Operator Theory 90 (2018), Paper No. 25, 44 pp.

\vspace{-0.3cm}

\bibitem{MazyaSobolev}
V. Maz'ya,
Sobolev Spaces with Applications to Elliptic Partial Differential
Equations,
second, revised and augmented edition,
Grundlehren der mathematischen Wissenschaften, 342,
Springer, Heidelberg, 2011.

\vspace{-0.3cm}

\bibitem{MazyaShubinDrum}
V. Maz'ya and M. Shubin,
Can one see the fundamental frequency of a drum?
Lett. Math. Phys. 74 (2005), 135--151.

\vspace{-0.3cm}

\bibitem{Necas1962}
J. Ne\v{c}as,
Sur une m\'ethode pour r\'esoudre les \'equations aux
d\'eriv\'ees partielles du type elliptique, voisine de la variationnelle,
Ann. Scuola Norm. Sup. Pisa (3) 16 (1962), 305--326.

\vspace{-0.3cm}

\bibitem{NecasDirectMethods}
J. Ne\v{c}as,
Direct Methods in the Theory of Elliptic Equations,
Springer Monographs in Mathematics,
Springer, Heidelberg, 2012.

\vspace{-0.3cm}

\bibitem{PortStone}
S. C. Port and C. J. Stone,
Brownian Motion and Classical Potential Theory,
Probability and Mathematical Statistics,
Academic Press, New York-London, 1978.

\vspace{-0.3cm}

\bibitem{ReedSimonI}
M. Reed and B. Simon, Methods of Modern Mathematical Physics.
I. Functional Analysis, second edition, Academic Press, New York, 1980.

\vspace{-0.3cm}

\bibitem{VazquezZuazua2000}
J. L. Vazquez and E. Zuazua,
The Hardy inequality and the asymptotic behaviour of the heat
equation with an inverse-square potential,
J. Funct. Anal. 173 (2000), 103--153.

\vspace{-0.3cm}

\bibitem{Wannebo1990}
A. Wannebo,
Hardy inequalities,
Proc. Amer. Math. Soc. 109 (1990), 85--95.

\vspace{-0.3cm}

\bibitem{ZiemerSobolev}
W. P. Ziemer,
Weakly Differentiable Functions: Sobolev Spaces and Functions of
Bounded Variation,
Graduate Texts in Mathematics 120,
Springer, New York, 1989.

\end{thebibliography}
\end{document}